\documentclass[a4paper]{article}

\usepackage[english]{babel}

\usepackage{cite}

\usepackage[utf8]{inputenc}
\usepackage[T1]{fontenc}
\usepackage{lmodern}

\usepackage{amssymb,amsmath,amsthm}

\usepackage{todonotes}
\usepackage{mathtools,bbm}
\usepackage[normalem]{ulem}
\usepackage{cancel}
\usepackage{verbatim}
\usepackage{tikz}
\usetikzlibrary{matrix}

\usepackage[shortlabels]{enumitem}
\setenumerate{label={\normalfont(\alph*)},topsep=4pt,itemsep=0pt} 

\usepackage{hyperref}

\usepackage{a4}
\usepackage{multirow}
\usepackage[footnotesize,bf,centerlast]{caption}
\usepackage[small,euler-digits,icomma,OT1,T1]{eulervm}

\usepackage{xcolor,colortbl}
\definecolor{Gray}{gray}{0.80}
\definecolor{LightGray}{gray}{0.90}

\newcommand{\cA}{\mathcal{A}}
\newcommand{\cB}{\mathcal{B}}

\newcommand{\cD}{\mathcal{D}}

\newcommand{\cF}{\mathcal{F}}

\newcommand{\cH}{\mathcal{H}}

\newcommand{\cK}{\mathcal{K}}
\newcommand{\cL}{\mathcal{L}}

\newcommand{\cP}{\mathcal{P}}
\newcommand{\cQ}{\mathcal{Q}}
\newcommand{\cR}{\mathcal{R}}
\newcommand{\cS}{\mathcal{S}}
\newcommand{\cT}{\mathcal{T}}
\newcommand{\cU}{\mathcal{U}}

\newcommand{\cX}{\mathcal{X}}
\newcommand{\cY}{\mathcal{Y}}

\newcommand{\fD}{\mathfrak{D}}

\newcommand{\fF}{\mathfrak{F}}

\newcommand{\fS}{\mathfrak{S}}

\newcommand{\bE}{\mathbb{E}}

\newcommand{\bP}{\mathbb{P}}
\newcommand{\bQ}{\mathbb{Q}}
\newcommand{\bR}{\mathbb{R}}

\newcommand{\PR}{\mathbb{P}}
\newcommand{\bONE}{\mathbbm{1}}

\newcommand{\dd}{ \mathrm{d}}

\DeclareMathOperator*{\LIM}{LIM} 
\DeclareMathOperator*{\subLIM}{subLIM}
\DeclareMathOperator*{\superLIM}{superLIM}
\DeclareMathOperator*{\LIMSUP}{LIM \, SUP}
\DeclareMathOperator*{\LIMINF}{LIM \, INF}

\renewcommand{\epsilon}{\varepsilon}

\newcommand{\vn}[1]{\left| \! \left| #1\right| \! \right|}

\newcommand{\ip}[2]{\langle #1,#2\rangle}

\numberwithin{equation}{section}

\newtheorem{theorem}{Theorem}[section]
\newtheorem{lemma}[theorem]{Lemma}
\newtheorem{proposition}[theorem]{Proposition}

\theoremstyle{definition}
\newtheorem{definition}[theorem]{Definition}

\newtheorem{remark}[theorem]{Remark}

\newtheorem{assumption}[theorem]{Assumption}

\newtheorem{condition}[theorem]{Condition}
\newtheorem{question}[theorem]{Question}

\title{The exponential resolvent of a Markov process and large deviations for Markov processes via Hamilton-Jacobi equations}

\author{Richard C. Kraaij\thanks{Delft Institute of Applied Mathematics, Delft University of Technology, Van Mourik Broekmanweg 6, 2628 XE Delft, The Netherlands. \emph{E-mail address}: r.c.kraaij@tudelft.nl}}
\date{\today}

\begin{document}

\maketitle

\begin{abstract}
	We study the Hamilton-Jacobi equation $f - \lambda Hf = h$, where $H f = e^{-f}Ae^f$ and where $A$ is an operator that corresponds to a well-posed martingale problem.
	
	We identify an operator that gives viscosity solutions to the Hamilton-Jacobi equation, and which can therefore be interpreted as the resolvent of $H$. The operator is given in terms of optimization problem where the running cost is a path-space relative entropy. 
	
	Finally, we use the resolvents to give a new proof of the abstract large deviation result of Feng and Kurtz \cite{FK06}.
	

\end{abstract}



\section{Introduction}

Let $E$ be Polish and let $A \subseteq C_b(E) \times C_b(E)$ be an operator such that the martingale problem for $A$ is well posed. In this paper, we study non-linear operator $H \subseteq C_b(E) \times C_b(E)$ given by all pairs $(f,g)$ such that
\begin{equation} \label{eqn:Hamiltonian_martingales}
t \mapsto \exp\left\{ f(X(t)) - f(X(0)) - \int_0^t g(X(s)) \dd s\right\}
\end{equation}
is a martingale with respect to $\cF_t := \sigma(X(s) \, | \, s \leq t)$ and where $X$ a solution of a well-posed martingale problem for $A$ (If $e^f \in \cD(A)$, then $(f,e^{-f}Ae^{f}) \in H$). 

The operator $H$, the martingales of \eqref{eqn:Hamiltonian_martingales} corresponding to $H$, and the semigroup 
\begin{equation} \label{eqn:intro_semigroup}
V(t)f(x) = \log \bE\left[e^{f(X(t))} \, \middle| \, X(0) =x\right].
\end{equation}
that formally correspond to $H$ play (possibly after rescaling) a key role in the theory of stochastic control and large deviations of Markov processes, see e.g. \cite{Fl77,Sh85,FlSo86,DuIiSo90,Pu97,Ac00,Pu01,Ac04,FK06}. 

Consider a sequence of Markov processes $X_n$. \cite{FK06} showed in their extensive monograph on the large deviations for Markov processes that the convergence of the non-linear semigroups $V_n(t)$ defined by $V_n(t)f(x) = \frac{1}{n}\log \bE\left[e^{nf(X_n(t))} \, \middle| \, X_n(0) =x\right]$ to some appropriate limiting semigroup $V(t)$ is a major step in establishing path-space large deviations for the sequence $X_n$.

It is well-known in the theory of linear semigroups that the convergence of semigroups $V_n(t)$ to $V(t)$ is essentially implied by the convergence of their infinitesimal generators `$H_n f = \partial_t V_n(t)f |_{t = 0}$' to  `$H f = \partial_t V(t)f |_{t = 0}$', see e.g. \cite{Tr58,Ka58,Ku74}. The results also hold for the non-linear context. However, in the non-linear setting, the relation between semigroup and generator is less clear. To be precise, $V(t)$ is generated by $H$ if we have a resolvent 
\begin{equation} \label{eqn:intro_naieve_resolvent}
R(\lambda) := (\bONE - \lambda H)^{-1}, \qquad \lambda > 0,
\end{equation}
which approximates the semigroup in the following way
\begin{equation}\label{eqn:intro_semigroup_approx_by_resolvent_linear}
V(t)h = \lim_m R\left(\frac{t}{m}\right)^m h, \qquad \forall \, h \in C_b(E), t \geq 0.
\end{equation}
To be able to effectively use the Trotter-Kato-Kurtz approximation results in the theory of large deviations or stochastic control, it is therefore important to have a grip on the resolvent that connects the semigroup $V(t)$ to the operator $H$ via \eqref{eqn:intro_naieve_resolvent} and \eqref{eqn:intro_semigroup_approx_by_resolvent_linear}.

\smallskip

An important first step in this direction was made in \cite{FK06} by replacing the Markov process $X$ by an approximating jump process with bounded generator. Indeed, in the case of bounded $A$ can establish the existence of \eqref{eqn:intro_naieve_resolvent} by using fix-point arguments. \cite{FK06} then proceed to establish path-space large deviations for sequences of Markov processes using probabilistic approximation arguments, semigroup convergence (Trotter-Kato-Kurtz) and the theory of viscosity solutions to characterize the limiting semigroup. 

\smallskip

A second observation is that in the context of diffusion processes, or for operators $H$ that are first-order, it is not clear that one can actually invert $(\bONE - \lambda H)$ due to issues with the domain: solutions of the Hamilton-Jacobi equation $f - \lambda H f = h$ can have non-differentiable points. However, one can often give a family of operators $R(\lambda)$ in terms of a deterministic control problem that yield viscosity solutions to the equation $f - \lambda Hf = h$. An extension $\widehat{H}$ of $H$ can then be defined in terms of $R$ such that the operator $\widehat{H}$ and the semigroup $V(t)$ are connected as in \eqref{eqn:intro_naieve_resolvent} and \eqref{eqn:intro_semigroup_approx_by_resolvent_linear}.

\smallskip

This paper therefore has a two-fold aim.
\begin{enumerate}[(1)]
	\item Identify an operator $R(\lambda)$ in terms of a control problem, which yields viscosity solutions to $f - \lambda Hf = h$ where $H$ is in terms of the martingales of \eqref{eqn:Hamiltonian_martingales}. This we aim to do in the context of \textit{general} (Feller) Markov processes. \label{item:intro_identify_resolvent}
	\item Give a new proof of the main large deviation result of \cite{FK06} by using the operators $R(\lambda)$. \label{item:intro_reprove_FK}
\end{enumerate}
Regarding \ref{item:intro_identify_resolvent}, we will show that the operators $R(\lambda)$ defined as
\begin{equation} \label{eqn:intro_resolvent}
R(\lambda)h(x) = \sup_{\bQ \in \cP(D_E(\bR^+))} \left\{ \int_0^\infty \left( \int h(X(t)) \bQ(\dd X) - S_t(\bQ \, | \, \PR_x ) \right) \tau_{\lambda}(\dd t)\right\}
\end{equation}
give viscosity solutions to the Hamilton-Jacobi equation for $H$. That is: $R(\lambda)h$ is a viscosity solution to
\begin{equation} \label{eqn:intro_HJ}
f - \lambda H f = h, \qquad h \in C_b(E), \lambda > 0.
\end{equation}
Here $S_t(\bQ \, | \PR_x)$ is the relative entropy of $\bQ$ with respect to the solution of the martingale problem started at $x$ evaluated up to time $t$, and $\tau_\lambda$ is the law of an exponential random variable with mean $\lambda$.

\smallskip


Our proof that $R(\lambda)$ is a viscosity solutions to the Hamilton-Jacobi equation will be carried out using a variant of a result by \cite{FK06} extended to an abstract context in \cite{Kr19}. The family $\{R(\lambda)\}_{\lambda > 0}$ of \eqref{eqn:intro_resolvent} gives viscosity solutions to \eqref{eqn:intro_HJ} if
\begin{enumerate}
	\item \label{item:intro_classical_inverse} for all $(f,g) \in H$ we have $R(\lambda)(f-\lambda g) = f$, 
	\item \label{item:intro_pseudo_resolvent} $R(\lambda)$ is contractive and a pseudo-resolvent. That is: $\vn{R(\lambda)}\leq 1$ and for all $h \in C_b(E)$ and $0 < \alpha < \beta$ we have 
	\begin{equation*}
	R(\beta)h = R(\alpha) \left(R(\beta)h - \alpha \frac{R(\beta)h - h}{\beta} \right).
	\end{equation*}
\end{enumerate}
In other words: if $R(\lambda)$ serves as a classical left-inverse to $\bONE- \lambda H$ and is also a pseudo-resolvent, then it is a viscosity right-inverse of $(\bONE- \lambda H)$.

To finish the analysis towards goal \ref{item:intro_identify_resolvent}, we need to establish that our resolvent approximates the semigroup:
\begin{enumerate}[resume]
	\item For the resolvent in \eqref{eqn:intro_resolvent} it holds that $V(t)h = \lim_m R\left(\frac{t}{m}\right)^m h$, where the semigroup is given by \eqref{eqn:intro_semigroup}. \label{item:intro_resolvent_approximation_semigroup}
\end{enumerate}
This result follows from the intuition that the sum of $n$ independent exponential random variables of mean $t/n$ converges to $t$. The difficulty lies in analysing the concatenation of suprema as in \eqref{eqn:intro_resolvent}, which will be carried out using suitable upper and lower bounds.

\smallskip

The second goal, \ref{item:intro_reprove_FK}, of this paper is to reprove the main large deviation result of \cite{FK06}. The general procedure is as follows:
\begin{itemize}
	\item Given exponential tightness, one can restrict the analysis to the finite-dimensional distributions.
	\item One establishes the large deviation principle for finite-dimensional distributions by assuming this is true at time $0$ and by proving that rescaled versions of the semigroups \eqref{eqn:intro_semigroup} of conditional log-moment generating functions converge.
	\item One proves convergence of the infinitesimal generators $H_n \rightarrow H$ and establishes well-posedness of the Hamilton-Jacobi equation $f - \lambda Hf = h$ to obtain convergence of the semigroups.
\end{itemize}
This paper follows the same general strategy, but establishes the third step in a new way. Instead of working with the resolvent of approximating Markov jump processes, the proof in this paper is based on a semigroup approximation argument of \cite{Kr19} combined with the explicit identification of the resolvents corresponding to the non-linear operators $H_n$.

\smallskip

We give a short comparison of the result in this paper to the main result in \cite{FK06}. Our condition on the convergence of Hamiltonians $H_n \rightarrow H$ is slightly simpler than the one in \cite{FK06}. This is due to being able to work with the Markov process itself instead of a approximating jump process. The result in this paper is a bit weaker in the sense that we assume the solutions to the martingale problems are continuous in the starting point, as opposed to only assuming measurability in \cite{FK06}. This is to keep the technicalities as simple as possible, and it is expected this can be generalized. In addition, \cite{FK06} establishes a result for discrete time processes, which we do not carry out here. This extension should be possible too.

\smallskip

The paper is organized as follows. We start in Section \ref{section:preliminaries} with preliminary definitions. In Section \ref{section:resolvent_and_HJ} we state the main results on the resolvent. In addition to the announced results \ref{item:intro_classical_inverse}, \ref{item:intro_pseudo_resolvent} and \ref{item:intro_resolvent_approximation_semigroup} we also obtain that $R(\lambda)$ is a continuous map on $C_b(E)$. Proofs of continuity of $R(\lambda)$ in addition to various other regularity properties are given in Section \ref{section:regularity}, the proofs of  \ref{item:intro_classical_inverse}, \ref{item:intro_pseudo_resolvent} and \ref{item:intro_resolvent_approximation_semigroup} are given in Section \ref{section:proofs_main_results}.

\smallskip

In Section \ref{section:Markov_LDP_simple} we state a simple version of the large deviation result. A more general version and its proof are given in Section \ref{section:Markov_LDP}.

\section{Preliminaries} \label{section:preliminaries}

Let $E$ be a Polish space. $C_b(E)$ denotes the space of continuous and bounded functions. Denote by $\cB(E)$ the Borel $\sigma$-algebra of $E$. Denote by $M(E)$ and $M_b(E)$ the spaces of measurable and bounded measurable functions $f : E \rightarrow [-\infty,\infty]$ and denote by $\cP(E)$ the space of Borel probability measures on $E$. $\vn{\cdot}$ will denote the supremum norm on $C_b(E)$. In addition to considering uniform convergence we consider the compact-open and strict topologies:
\begin{itemize}
	\item The \textit{compact open} topology $\kappa$ on $C_b(E)$ is generated by the semi-norms $p_K(f) = \sup_{x \in K} |f(x)|$, where $K$ ranges over all compact subsets of $E$.
	\item The \textit{strict} topology $\beta$ on $C_b(E)$ is generated by all semi-norms
	\begin{equation*}
	p_{K_n,a_n}(f) := \sup_n a_n \sup_{x \in K_n} |f(x)|
	\end{equation*}
	varying over non-negative sequences $a_n$ converging to $0$ and sequences of compact sets $K_n \subseteq E$. See e.g. \cite{Se72,Ca11,Kr19d}. 
\end{itemize}
As we will often work with the convergence of sequences for the strict topology, we characterize this convergence and give a useful notion of taking closures.

A sequence $f_n$ converges to $f$ for the strict topology if and only if $f_n$ converges to $f$ bounded and uniformly on compacts (buc): 
\begin{equation*}
\sup_n \vn{f_n} < \infty, \qquad \forall \text{ compact } K \subseteq E: \quad \lim_{n \rightarrow \infty} \sup_{x \in K} \left| f_n(x) - f(x)\right| = 0.
\end{equation*} 

Let $B_r \subseteq C_b(E)$ be the collection of functions $f$ such that $\vn{f} \leq r$. We say that $\widehat{D}$ is the \textit{quasi-closure} of $D \subseteq C_b(E)$ if $\widehat{D} = \bigcup_{r > 0} \widehat{D}_r$, where $\widehat{D}_r$ is the strict closure of $D \cap B_r$.

\smallskip

We denote by $D_E(\bR^+)$ the \textit{Skorokhod space} of trajectories $X : \bR^+ \rightarrow E$ that have left limits and are right-continuous. We equip this space with its usual topology, see \cite[Chapter 3]{EK86}. As $D_E(\bR^+)$ is our main space of interest, we write $\cP := \cP(D_E(\bR^+))$.

\smallskip

Let $\cX$ be a general Polish space (e.g. $E$ or $D_E(\bR^+)$). For two measures $\mu,\nu \in \cP(\cX)$ we denote by
\begin{equation*}
S(\nu \, | \, \mu) = 
\begin{cases}
\int \log \frac{\dd \nu}{\dd \mu} \, \dd \nu & \text{if } \nu << \mu, \\
\infty & \text{otherwise},
\end{cases}
\end{equation*}
the \textit{relative entropy} of $\nu$ with respect to $\mu$. For any sub-sigma algebra $\cF$ of $\cB(\cX)$, we denote by $S_\cF$ the relative entropy when the measures are restricted to the $\sigma$-algebra $\cF$. In the text below, we will often work with the space $D_E(\bR^+)$. We will then write $S_t$ for the relative entropy when we restrict to $\cF_t := \sigma\left(X(s) \, | \, s \leq t \right)$.

Finally, for $\lambda > 0$, denote by $\tau_\lambda \in \cP(\bR^+)$ the law of an exponential random variable with mean $\lambda$:
\begin{equation*}
\tau_\lambda(\dd t) = \bONE_{\{t \geq 0\}} \lambda^{-1} e^{-\lambda^{-1} t} \dd t.
\end{equation*}

\subsection{The martingale problem}

\begin{definition}[The martingale problem]
	Let $A \colon \cD(A) \subseteq C_b(E) \rightarrow C_b(E)$ be a linear operator. For $(A,\cD(A))$ and a measure $\nu \in \cP(E)$, we say that $\PR \in \cP(D_E(\bR^+))$ solves the \textit{martingale problem} for $(A,\nu)$ if $\PR \, \circ \, X(0)^{-1} = \nu$ and if for all $f \in \cD(A)$ 
	\begin{equation*}
	M_f(t) := f(X(t)) - f(X(0)) - \int_0^t Af(X(s)) \dd s
	\end{equation*}
	is a martingale with respect to its natural filtration $\cF_t := \sigma\left(X(s) \, | \, s \leq t \right)$ under $\PR$. 
	
	\smallskip
	
	 We say that \textit{uniqueness} holds for the martingale problem if for every $\nu \in \cP(\cX)$ the set of solutions of the martingale problem that start at $\nu$ has at most one element. Furthermore, we say that the martingale problem is \textit{well-posed} if this set contains exactly one element for every $\nu$.
\end{definition}

\subsection{Viscosity solutions of Hamilton-Jacobi equations}

Consider an operator $B \subseteq C_b(E) \times C_b(E)$. If $B$ is single valued and $(f,g) \in B$, we write $Bf := g$. We denote $\cD(B)$ for the domain of $B$ and $\cR(B)$ for the range of $B$.

\begin{definition}
	Let $B \subseteq C_b(E) \times C_b(E)$. Fix $h \in C_b(E)$ and $\lambda > 0$. Consider the equation
	\begin{equation} \label{eqn:HJ_viscosity_solution_def}
	f - \lambda B f = h
	\end{equation}
	\begin{itemize}
		\item We say that a bounded upper semi-continuous function $u : E \rightarrow \bR$ is a \textit{subsolution} of equation \eqref{eqn:HJ_viscosity_solution_def} if for all $(f,g) \in B$  there is a sequence $x_n \in E$ such that
		\begin{equation} \label{eqn:subsol_optimizing_sequence_simple}
		\lim_{n \rightarrow \infty} u(x_n) - f(x_n)  = \sup_x u(x) - f(x),
		\end{equation}
		and
		\begin{equation} \label{eqn:subsol_sequence_outcome_simple}
		\limsup_{n \rightarrow \infty} u(x_n) - g(x_n) - h(x_n) \leq 0.
		\end{equation}
		\item We say that a bounded lower semi-continuous function $v : E \rightarrow \bR$ is a \textit{supersolution} of equation \eqref{eqn:HJ_viscosity_solution_def} if for all $(f,g) \in B$  there is a sequence $x_n \in E$ such that
		\begin{equation} \label{eqn:supersol_optimizing_sequence_simple}
		\lim_{n \rightarrow \infty} v(x_n) - f(x_n)  = \inf_x v(x) - f(x),
		\end{equation}
		and
		\begin{equation} \label{eqn:supersol_sequence_outcome_simple}
		\liminf_{n \rightarrow \infty} v(x_n) - g(x_n) - h(x_n) \geq 0.
		\end{equation}
		\item We say that $u$ is a \textit{solution} of \eqref{eqn:HJ_viscosity_solution_def}  if it is both a subsolution and a supersolution.
		\item We say that \eqref{eqn:HJ_viscosity_solution_def} satisfies the \textit{comparison principle} if for every subsolution $u$, we have
		\begin{equation} \label{eqn:comparison_estimate_simple}
		\sup_x u(x) - v(x) \leq 0.
		\end{equation}
	\end{itemize}
\end{definition}
Note that the comparison principle implies uniqueness of viscosity solutions.

\subsection{Convergence of operators}

\begin{definition} \label{definition_extended_limit}
	For a sequence of operators $B_n \subseteq C_b(E) \times C_b(E)$ and $B \subseteq C_b(E) \times C_b(E)$ we say that
	$B$ is subset of the \textit{extended limit} of $B_n$, denoted by $B \subseteq ex-\LIM B_n$ if for all $(f,g) \in B$ there are $(f_n,g_n) \in B_n$ such that $\beta-\lim_n f_n = f$ and $\beta-\lim g_n = g$.
\end{definition}

\subsection{Large deviations}

\begin{definition}
	Let $\{X_n\}_{n \geq 1}$ be a sequence of random variables on a Polish space $\cX$. Furthermore, consider a function $I : \mathcal{X} \rightarrow [0,\infty]$ and a sequence $\{r_n\}_{n \geq 1}$ of positive real numbers such that $r_n \rightarrow \infty$. We say that
	\begin{itemize}
		\item  
		the function $I$ is a \textit{rate-function} if the set $\{x \, | \, I(x) \leq c\}$ is closed for every $c \geq 0$. We say $I$ is good if the sub-level sets are compact.
		\item 
		the sequence $\{X_n\}_{n\geq 1}$ is \textit{exponentially tight} at speed $r_n$ if, for every $a \geq 0$, there exists a compact set $K_a \subseteq \mathcal{X}$ such that $\limsup_n r_n^{-1} \log \, \PR[X_n \notin K_a] \leq - a$.
		\item 
		the sequence $\{X_n\}_{n\geq 1}$ satisfies the \textit{large deviation principle} with speed $r_n$ and good rate-function $I$ if for every closed set $A \subseteq \mathcal{X}$, we have 
		\begin{equation*}
		\limsup_{n \rightarrow \infty} \, r_n^{-1} \log \PR[X_n \in A] \leq - \inf_{x \in A} I(x),
		\end{equation*}
		and if for every open set $U \subseteq \mathcal{X}$, 
		\begin{equation*}
		\liminf_{n \rightarrow \infty} \, r_n^{-1} \log \PR[X_n \in U] \geq - \inf_{x \in U} I(x).
		\end{equation*}
	\end{itemize}
\end{definition}

\section{The non-linear resolvent of a Markov process} \label{section:resolvent_and_HJ}

Our main result is based on the assumption that the martingale problem is well-posed and that the solution map in terms of the starting point is continuous.

\begin{condition} \label{condition:Markov_single}
	$A \subseteq C_b(E) \times C_b(E)$ is an operator such that the martingale problem for $A \subseteq C_b(E) \times C_b(E)$ is well-posed. Denote by $\PR_x \in D_E(\bR^+)$ the solution that satisfies $X(0) = x$, $\PR_x$ almost surely. The map $x \mapsto \PR_x$ is assumed to be continuous for the weak topology on $\cP = \cP(D_E(\bR^+))$.
	
\end{condition}

We introduce the triplet of key objects in semi-group theory: generator, resolvent, and semigroup.

\begin{definition} \label{definition:key_definitions_H_resolvent_semigroup}
	
	\begin{enumerate}
		\item Let $H$ be a collection of pairs $(f,g) \in C_b(E) \times C_b(E)$ such that for all $x \in E$ 
		\begin{equation*}
		t \mapsto \exp\left\{f(X(t)) - f(X(0)) - \int_0^t g(X(s)) \dd s \right\}
		\end{equation*}
		are martingales with respect to the filtration $\cF_t := \sigma\left(X(s) \, | \, s \leq t \right)$ and law $\PR_x$.
		\item For $\lambda > 0$ and $h \in C_b(E)$, define
		\begin{equation*}
		R(\lambda)h(x) = \sup_{\bQ \in \cP} \left\{ \int_0^\infty \left( \int h(X(t)) \bQ(\dd X) - S_t(\bQ \, | \, \PR_x ) \right) \tau_{\lambda}(\dd t)\right\}.
		\end{equation*}
		\item For $t \geq 0$ and $h \in C_b(E)$, define
		\begin{equation*}
		V(t)f(x) = \log \bE_x\left[e^{h(X(t))}\right] = \sup_{\bQ \in \cP} \left\{ \int h(X(t)) \bQ(\dd X) - S(\bQ \, | \, \PR_x) \right\}.
		\end{equation*}
		Note that the final equality follows by Lemma \ref{lemma:DV_variational}.
	\end{enumerate}
\end{definition}

The following is an immediate consequence of \cite[Lemma 4.3.2]{EK86}.

\begin{lemma}
	Let Condition \ref{condition:Markov_single} be satisfied. We have
	\begin{equation*}
	\left\{(f, e^{-f}g) \, \middle| \, (e^f,g) \in A\right\} \subseteq H. 
	\end{equation*}
\end{lemma}

The first main result of this paper is the following.

\begin{theorem} \label{theorem:Markov_solving_HJ}
	Let Condition \ref{condition:Markov_single} be satisfied. For each $h \in C_b(E)$ and $\lambda > 0$ the function $R(\lambda)h$ is a viscosity solution to $f - \lambda H f = h$.
\end{theorem}

The proof of this result follows in Section \ref{section:proofs_main_results}. To facilitate further use of the non-linear resolvent, we establish also that
\begin{enumerate}
	\item The map $R(\lambda)$ maps $C_b(E)$ into $C_b(E)$.
	\item The operators $R(\lambda)$ act as the resolvent of the semigroup $\{V(t)\}_{t \geq 0}$.
\end{enumerate}

These properties will allow us to use our main result to establish large deviations in a later part of the paper, see Section \ref{section:Markov_LDP}. We state (a) and (b) as Propositions.

\begin{proposition} \label{proposition:continuity_of_resolvent}
	For every $\lambda > 0$ and $h \in C_b(E)$ we have $R(\lambda) h \in C_b(E)$.
\end{proposition}

\begin{proposition} \label{proposition:resolvents_approximate_semigroup}
	For each $h \in C_b(E)$, $t \geq 0$ and $x \in E$ we have
	\begin{equation*}
	\lim_{m \rightarrow \infty} R\left(\frac{t}{m}\right)^m h = V(t)h.
	\end{equation*}
	for the strict topology.
\end{proposition}

Proposition \ref{proposition:continuity_of_resolvent} will be verified in Section \ref{section:regularity}, in which we will also verify other regularity properties of $R(\lambda)$. Proposition \ref{proposition:resolvents_approximate_semigroup} is a part of our main results connecting the resolvent and semigroup and will be established in Section \ref{section:proofs_main_results}.

\subsection{Strategy of the proof of Theorem \ref{theorem:Markov_solving_HJ} and discussion on extensions} \label{section:discussion_proofs}

Theorem \ref{theorem:Markov_solving_HJ} will follow as a consequence of Proposition 3.4 of \cite{Kr19}. We therefore have to check three properties of $R(\lambda)$:
\begin{enumerate}
	\item For all $(f,g) \in H$, we have $f = R(\lambda)(f - \lambda g)$ ;
	\item The pseudo-resolvent property: for all $h \in C_b(E)$ and $0 < \alpha < \beta$ we have 
	\begin{equation*}
	R(\beta)h = R(\alpha) \left(R(\beta)h - \alpha \frac{R(\beta)h - h}{\beta} \right).
	\end{equation*}
	\item $R(\lambda)$ is contractive. 
\end{enumerate}

We verify (c) in Section \ref{section:regularity} as it relates to the regularity of the resolvent. We verify (a) and (b) in Sections \ref{subsection:Markov_invert} and \ref{subsection:Markov_pseudoresolvent} respectively. 

As is known from the theory of weak convergence, the resolvent is related to exponential integrals. 
\begin{itemize}
	\item (a) is related to integration by parts: for bounded measurable functions $z$ on $\bR^+$, we have
	\begin{equation*}
	\lambda \int_0^\infty  z(t) \, \tau_\lambda( \dd t) = \int_0^\infty \int_0^t z(s) \, \dd s \, \tau_\lambda(\dd t).
	\end{equation*}
	\item (b) is related to a more elaborate property of exponential random variables. Let $0 < \alpha < \beta$ then
	\begin{equation*}
	\int_0^\infty z(s) \tau_\beta(\dd s) = \frac{\alpha}{\beta} \int_0^\infty z(s) \tau_\alpha(\dd s) + \left(1 - \frac{\alpha}{\beta}\right) \int_0^\infty \int_0^\infty z(s+u) \, \tau_\beta(\dd u) \, \tau_\alpha(\dd s).
	\end{equation*}
	\item Finally, the approximation property of Proposition \ref{proposition:resolvents_approximate_semigroup} is essentially a law of large numbers. The sum of $n$ independent random variables of mean $t/n$ converges to $t$.
\end{itemize}
 In the non-linear setting, our resolvent is given in terms of an optimization problem over an exponential integral. Thus, our method is aimed towards treating the optimisation procedures by careful choices of measures and decomposition and concatenation or relative entropies by using Proposition \ref{proposition:relative_entropy_decomposition} and then using the properties of exponential integrals.

\smallskip

Any of the results mentioned in the above section can be carried out by introducing an extra scaling parameter into the operators.

\begin{remark} \label{remark:rescaling_of_operators}
	Fix $r > 0$. Let $H[r]$ be a collection of pairs $(f,g)$ such that
	\begin{equation*}
	t \mapsto \exp\left\{r\left(f(X(t)) - f(X(0)) - \int_0^t g(X(s)) \dd s \right) \right\}
	\end{equation*}
	are martingales. As above, we have $\left\{(f, r^{-1} e^{-rf}g) \, \middle| \, (e^{rf},g) \in A \right\} \subseteq H[r]$. Relatively straightforwardly, chasing the constant $r$, one can show that $R[r](\lambda)h$ gives viscosity solutions to $f - \lambda H[z]f = h$, where
	\begin{equation*}
	R[z](\lambda)h(x) = \sup_{\bQ \in \cP} \left\{ \int_0^\infty \lambda^{-1} e^{-\lambda^{-1}t} \left( \int h(X(t)) \bQ(\dd X) - \frac{1}{r} S_t(\bQ \, | \, \PR_x ) \right) \dd t\right\}.
	\end{equation*}
	We also have
	\begin{equation*}
	\lim_{m \rightarrow \infty} R[r]\left(\frac{t}{m}\right)^m h(x) = r^{-1} V(t)(rh)(x).
	\end{equation*}
\end{remark}

\begin{question} \label{question:upper_and_lower_bound}
	To some extent one could wonder whether Theorem \ref{theorem:Markov_solving_HJ} has an extension where $H_\dagger$ is a collection of pairs $(f,g)$ such that
	\begin{equation*}
	t \mapsto \exp\left\{f(X(t)) - f(X(0)) - \int_0^t g(X(s)) \dd s \right\}
	\end{equation*}
	are supermartingales, and where $H_\ddagger$ is a collection of pairs $(f,g)$ such that
	\begin{equation*}
	t \mapsto \exp\left\{f(X(t)) - f(X(0)) - \int_0^t g(X(s)) \dd s \right\}
	\end{equation*}
	are submartingales.
	
	The statement would become that for each $h \in C_b(E)$ and $\lambda > 0$ the function $R(\lambda)h$ is a viscosity subsolution to $f - \lambda H_\dagger f = h$ and a viscosity supersolution to $f - \lambda H_\ddagger f = h$. 
	
	Indeed, some of the arguments in Section \ref{section:proofs_main_results} can be carried out for sub- and super-martingales respectively. Certain arguments, however,use that we work with martingales. For example, Lemma \ref{lemma:DV_variational} holds for probability measures only.
\end{question}

\section{Large deviations for Markov processes} \label{section:Markov_LDP_simple}

In this section, we consider the large deviations on $D_E(\bR^+)$ of a sequence of Markov processes $X_n$. In Section \ref{section:Markov_LDP} below, we will instead consider the more general framework where the $X_n$ take their values in a sequence of spaces $E_n$ that are embedded in $E$ by a map $\eta_n$ and where the images $\eta_n(E_n)$ converge in some appropriate way to $E$. As this introduces a whole range of technical complications, we restrict ourselves in this section to the most simple case.

\begin{condition} \label{condition:simple_markov_ldp}
	Let $A_n \subseteq C_b(E) \times C_b(E)$ be linear operators and let $r_n$ be positive real numbers such that $r_n \rightarrow \infty$. Suppose that
	\begin{itemize}
		\item The martingale problems for $A_n$ are well-posed. Denote by $x \mapsto \PR_x^n$ the solution to the martingale problem for $A_n$.
		\item For each $n$ that $x \mapsto \PR_x^n$ is continuous for the weak topology on $\cP(D_E(\bR^+))$.
		\item for all compact sets $K \subseteq E$ and $a \geq 0$ there is a compact set $K_a \subseteq D_E(\bR^+)$ such that
		\begin{equation*}
		\limsup_n \sup_{x \in K} \frac{1}{r_n} \log \PR_x^n[K_a^c] \leq -a.
		\end{equation*}
	\end{itemize}
\end{condition}

The first two conditions correspond to Condition \ref{condition:Markov_single}. The final one states that we have exponential tightness of the processes $X_n$ uniformly in the starting position in a compact set.

\smallskip

Corresponding to the previous section, define the operators $H_n$ consisting of pairs $(f,g) \in C_b(E) \times C_b(E)$ such that
\begin{equation*}
t \mapsto \exp\left\{r_n \left( f(X_n(t)) - f(X_n(0)) - \int_0^t g(X_n(s)) \dd s \right)  \right\}
\end{equation*}
are martingales. Also define the rescaled log moment-generating functions
\begin{equation*}
V_n(t)f(y) := \frac{1}{r_n} \log \bE\left[e^{r_n f(X_n(t))} \, \middle| \, X_n(0) = x\right].
\end{equation*}

\begin{theorem} \label{theorem:LDP_simple}
	Let Condition \ref{condition:simple_markov_ldp} be satisfied. Let $r_n > 0$ be some sequence such that $r_n \rightarrow \infty$. Suppose that
	\begin{enumerate}
		\item The large deviation principle holds for $X_n(0)$ on $E$ with speed $r_n$ and good rate function $I_0$.
		\item There is an operator $H \subseteq ex-\LIM H_n$.
		\item The comparison principle holds for $f - \lambda H f = h$ for all $h \in C_b(E)$ and $\lambda > 0$.
	\end{enumerate}
	Then there exists a semigroup $V(t)$ on $C_b(E)$ such that if $\beta-\lim f_n = f$ and $t_n \rightarrow t$ if holds that $\beta-\lim V(t_n)f_n = V(t)f$.
	
	In addition, the processes $X_n$ satisfy a large deviation principle on $D_E(\bR^+)$ with speed $r_n$ and rate function
	\begin{equation} \label{eqn:LDP_rate_simple}
	I(\gamma) = I_0(\gamma(0)) + \sup_{k \geq 1} \sup_{\substack{0 = t_0 < t_1 < \dots, t_k \\ t_i \in \Delta_\gamma^c}} \sum_{i=1}^{k} I_{t_i - t_{i-1}}(\gamma(t_i) \, | \, \gamma(t_{i-1})).
	\end{equation}
	Here $\Delta_\gamma^c$ is the set of continuity points of $\gamma$. The conditional rate functions $I_t$ are given by
	\begin{equation*}
	I_t(y \, | \, x) = \sup_{f \in C_b(E)} \left\{f(y) - V(t)f(x) \right\}.
	\end{equation*}
\end{theorem}

\begin{remark}
	A representation for $I$ in a Lagrangian form can be obtained by the analysis in Chapter 8 of \cite{FK06}. To some extent the analysis is similar to the one of this paper. First, one identifies the resolvent as a deterministic control problem by showing that it solves the Hamilton-Jacobi equation in the viscosity sense. Second, one shows that it approximates a control-semigroup. Third, one uses the control-semigroup to show that \eqref{eqn:LDP_rate_simple} is also given in terms of the control problem. 
\end{remark}

\section{Regularity of the semigroup and resolvent} \label{section:regularity}

The main object of study of this paper is the resolvent introduced in Definition \ref{definition:key_definitions_H_resolvent_semigroup}. Before we start with the main results, we first establish that the resolvent itself is `regular':
\begin{itemize}
	\item We establish that $R(\lambda) h \in C_b(E)$ for $\lambda > 0$ and $h \in C_b(E)$.
	\item We establish that $h \mapsto R(\lambda)h$ is sequentially continuous for the strict topology.
	\item We establish that $\lim_{\lambda \downarrow 0} R(\lambda)h = h$ for the strict topology.
\end{itemize}
Before starting with analysing the resolvent, we establish regularity properties for the cost function that appears in the definition of $R(\lambda)$.

\subsection{Properties of relative entropy} \label{subsection:regularity_of_entropy}

A key property of Legendre transformation is that convergence of convex functionals implies (and is often equivalent) to Gamma convergence of their convex duals. This can be derived from a paper of Zabell \cite{Za92a}. In the context of weak convergence of measures this has recently been established with a direct proof by Mariani in Proposition 3.2 of \cite{Ma18}.

We state the result for completeness.

\begin{proposition} \label{proposition:Gamma_convergence_relative_entropy}
	Let $X$ be some Polish space.
		Then (a) and (b) are equivalent:
	\begin{enumerate}
		\item $\mu_n \rightarrow \mu$ weakly,
		\item The functionals $S(\cdot \, | \, \mu_n)$ Gamma converge to $S(\cdot \, | \, \mu)$: that is:
		\begin{enumerate}[(1)]
			\item The Gamma lower bound: for any sequence $\nu_n \rightarrow \nu$ we have $\liminf_n S(\nu_n \, | \, \mu_n) \geq S(\nu \, | \, \mu)$.
			\item The Gamma upper bound: for any $\nu$ there are $\nu_n$ such that $\nu_n \rightarrow \nu$ such that $\limsup_n S(\nu_n \, | \, \mu_n) \leq S(\nu \, | \, \mu)$.
		\end{enumerate} 
	\end{enumerate}
\end{proposition}

Our resolvent is given in terms of the cost functional
\begin{equation} \label{eqn:definition_exponential_entropy}
\cS_\lambda(\bQ \, | \, \PR_x) := \int_0^\infty S_t(\bQ \, | \, \PR_x) \tau_\lambda(\dd t).
\end{equation}
Below, we establish Gamma convergence for $\cS_\lambda$. 
\begin{itemize}
	\item The Gamma $\liminf_n$ inequality, in addition to the compactness of the level sets (coercivity) of $\cS_\lambda$ is, established in Lemma \ref{lemma:compact_level_sets_integrated_relative_entropy}.
	\item In Proposition \ref{proposition:uniform_compact_levelsets} we strengthen the coercivity to allow for compactness of the level sets of $\cS_\lambda$ uniformly for small $\lambda$ (equi-coercivity). This property will allow us to study $R(\lambda)$ uniformly for small $\lambda$.
	\item The Gamma $\limsup_n$ inequality is established in Proposition \ref{proposition:gamma_convergence_integral_entropy}.
\end{itemize}

\subsubsection{The $\Gamma-\liminf$ inequality and coercivity}

\begin{lemma} \label{lemma:compact_level_sets_integrated_relative_entropy}
	For any $\lambda >0$ the map
	\begin{equation*}
	(\PR,\bQ) \mapsto \cS_\lambda(\bQ \, | \, \PR) = \int_0^\infty S_t(\bQ \, | \, \PR) \tau_\lambda(\dd t)
	\end{equation*}
	is lower semi-continuous. In addition, the map has compact sublevel sets in the following sense: fix a compact set $K \subseteq \cP(D_E(\bR^+))$ and $c \geq 0$.  Then the set
	\begin{equation*}
	A(c) := \left\{\bQ \in \cP(D_E(\bR^+)) \, \middle| \, \exists \, \PR \in K: \, \cS_\lambda(\bQ \, | \, \PR) \leq c \right\}
	\end{equation*}
	is compact.
\end{lemma}

\begin{proof}
	The first claim follows by lower semi-continuity of $(\PR,\bQ) \mapsto S_t(\bQ \, | \, \PR)$ and Fatou's lemma. For the second claim note that a set $A \subseteq \cP(D_E(\bR^+))$ is compact if the set of measures in $A$ restricted to $\cP(D_E([0,t)))$ is compact for all $t$, see Theorem 3.7.2 in \cite{EK86}.

	Thus, fix $t$ and suppose $\bQ \in A(c)$. Then there is some $\PR \in K$ such that
	\begin{align*}
	S_t(\bQ \, | \, \bP) & \leq \int_0^\infty S_{t+s}(\bQ \, | \, \PR) \tau_\lambda(\dd s) \\
	& =  e^{\lambda^{-1} t} \int_t^\infty  S_u(\bQ \, | \, \PR) \tau_\lambda(\dd u)  \\
	& \leq e^{\lambda^{-1} t} c .
	\end{align*}
	The result now follows by Proposition \ref{proposition:equi_coercivity_relative_entropy}.
\end{proof}

The final estimate in the above proof is not uniform for small $\lambda$. this is due to the fact that the exponential random variables $\tau_\lambda$ concentrate near $0$. Thus, we can only control the relative entropies for small intervals of time after which the measure $\bQ$ is essentially free to do what it wants. Equi-coercivity of the level sets can be recovered to some extent by restricting the interval on which one is allowed to tilt the measure.

\begin{proposition} \label{proposition:uniform_compact_levelsets}
	Fix a compact set $K \subseteq \cP$, $\lambda_0 > 0$ and constants $c \geq 0$ and $\varepsilon \in (0,1)$. Let $T(\lambda) := -\lambda \log \varepsilon$. Then the set
	\begin{equation*}
	\bigcup_{0 < \lambda \leq \lambda_0} \bigcup_{\PR \in K} \left\{ \bQ \in \cP \, \middle| \, \cS_\lambda(\bQ \, | \, \PR) \leq c, \quad S_{T(\lambda)}(\bQ \, | \, \PR) = S(\bQ \, | \, \PR) \right\}
	\end{equation*}
	is relatively compact in $\cP$.	
\end{proposition}

\begin{proof}
	First recall that a set of measures in $\cP$ is compact if the set of their restrictions to a finite time interval is relatively compact. 
	
	Pick $\PR \in K$ and $0 < \lambda \leq \lambda_0$ and let $\bQ^* \in \cP$ be such that $\cS_\lambda(\bQ^* \, | \, \PR) \leq c$. We obtain
	\begin{align*}
	S_{T(\lambda)}(\bQ^* \, | \PR) & = \int_0^\infty S_{T(\lambda)+s}(\bQ^* \, | \, \PR) \tau_\lambda(\dd s) \\
	& = e^{\lambda^{-1} T(\lambda)} \int_{T(\lambda)}^\infty S_{u}(\bQ^* \, | \, \PR) \tau_\lambda(\dd u) \\
	& \leq e^{\lambda^{-1} T(\lambda)} c = \frac{c}{\varepsilon}
	\end{align*}
	which is uniformly bounded in $\lambda$. Note that as $S_{T(\lambda)}(\bQ^* \, | \PR) = S(\bQ^* \, | \, \PR)$, we have for all $t \geq 0$ that $S_{T(\lambda)}(\bQ^* \, | \PR) = S_{T(\lambda)+t}(\bQ^* \, | \, \PR)$. The map $\lambda \mapsto T(\lambda)$ is increasing, so if $t \geq T(\lambda_0)$ then $t \geq T(\lambda)$ and $S_t(\bQ^* \, | \, \PR) = S_{T(\lambda)}(\bQ^* \, | \, \PR)$. This implies that the measure $\bQ^*$ is contained in the set
	\begin{equation*}
	\bigcup_{\PR \in K} \left\{ \bQ \in \cP \, \middle| \, \forall t \geq T(\lambda_0): \,  \, S_t(\bQ \, | \, \PR) = S_{T(\lambda)}(\bQ^* \, | \PR) \leq \frac{c}{\varepsilon} \right\}.
	\end{equation*}
	By the remark at the start of the proof, this set is compact by Proposition \ref{proposition:equi_coercivity_relative_entropy}. 
\end{proof}

\subsubsection{The $\Gamma-\limsup$ inequality: construction of a recovery sequence}

For the proof of the $\Gamma-\liminf$ inequality, we could use Proposition \ref{proposition:Gamma_convergence_relative_entropy} and Fatou. In the context of the $\Gamma-\limsup$ inequality, we run into the following issue.

Given a sequence $x_n \rightarrow x$ and fixed time $t$, the result of Proposition \ref{proposition:Gamma_convergence_relative_entropy} will allow to construct a sequence $\bQ_n$ converging to $\bQ$ such that $\limsup_n S_t(\bQ_n \, | \, \PR_{x_n}) \leq S_t(\bQ \, | \, \PR_x)$. This statement can, however, not immediately be lifted to the functional $\cS_\lambda$ as the construction gives no information on times $s \neq t$.

But, using the Markovian structure of the family $\{\PR_y\}_{y \in E}$ and continuity of these measures in $y$ will allow us to construct measures $\bQ_n$ converging to $\bQ$ such that also $\limsup_n \cS_\lambda(\bQ_n \, | \, \PR_{x_n}) \leq \cS_\lambda (\bQ \, | \, \PR_x)$. This construction will be carried out via a projective limit argument.

	\begin{proposition} \label{proposition:gamma_convergence_integral_entropy}
		Let $\bQ$ be such that $\cS_\lambda(\bQ \, | \, \PR_x) = \int_0^\infty S_t(\bQ \, | \, \PR_x) \tau_\lambda(\dd t) < \infty$.
		
		Then, there are measures $\bQ_n \in \cP(D_E(\bR^+))$ that converge to $\bQ$. In addition
		\begin{align*}
		S_t(\bQ_n \, | \, \PR_{x_n}) & \leq S_t(\bQ \, | \, \PR_{x}) + 1, & \forall \, n, \forall \, t \\
		\limsup_n S_t(\bQ_n \, | \, \PR_{x_n}) & \leq S_t(\bQ \, | \, \PR_x), & \forall \, t.
		\end{align*}
		We infer from Fatou's lemma that also
		\begin{equation*}
		\limsup_{n \rightarrow \infty} \cS_\lambda(\bQ_n \, | \, \PR_{x_n}) = \limsup_{n \rightarrow \infty}\int_0^\infty S_t(\bQ_n \, | \, \PR_{x_n}) \tau_\lambda(\dd t) \leq \int_0^\infty S_t(\bQ \, | \, \PR_x) \tau_\lambda(\dd t) = \cS_\lambda(\bQ \, | \, \PR_x).
		\end{equation*}
		
	\end{proposition}

We will construct the measures $\bQ_n$ by arguing via appropriately chosen finite-dimensional projections of $\bQ$. Thus, we need to establish a conditional version of the $\limsup_n$ inequality for Gamma convergence of relative entropy functionals. We state and prove this conditional result first, after which we prove Proposition \ref{proposition:gamma_convergence_integral_entropy}.

\begin{lemma} \label{lemma:conditional_gamma_entropy}
	Let $\cX,\cY$ be Polish spaces. Let $\nu, \mu \in \cP(\cX \times \cY)$ and suppose that $\mu_n$ are measures on $\cX \times \cY$ converging to $\mu$. 
	Denote by $\mu_{n,0}, \mu_0,\nu_0 \in \cP(\cX)$ the restrictions of $\mu_n,\mu,\nu$ to $\cX$.
	
	Suppose that 
	\begin{enumerate}
		\item There are measures $\nu_{0,n}$ on $\cX$ such that $\nu_{n,0}$ converges weakly to $\nu_0$ and such that $\limsup_{n \rightarrow \infty} S_{\cX}(\nu_{n,0} \, | \, \mu_{n,0}) \leq S_{\cX}(\nu_0 \, | \, \mu_0)$.
		\item Suppose there is a family of measures $\{\hat{\mu}(\cdot \, | \, x)\}_{x \in \cX}$ on $\cY$ that is weakly continuous in $x$. Suppose that this family of measures is a version of the regular conditional measures  $\mu_n(\cdot \, | \, x)$ and also of $\{\mu(\cdot \, | \, x)\}_{x \in \cX}$. 
	\end{enumerate}
	Then there are measures $\nu_n \in \cP(\cX \times \cY)$ converging to $\nu$ such that the restriction of $\nu_n$ to $\cX$ equals $\nu_{n,0}$ and $\limsup_{n \rightarrow \infty} S(\nu_n \, | \, \mu_n) \leq S(\nu \, | \, \mu)$.
\end{lemma}

	\begin{proof}
		First of all, note that if $S(\nu \, | \, \mu) = \infty$, the proof is trivial. Thus, assume $S(\nu \, | \, \mu) < \infty$.
		
		Denote by $\nu(\cdot \, | \, x)$ a version of the regular conditional probability of $\nu$ conditional on $x \in \cX$. By the Skorokhod representation theorem, \cite[Theorem 8.5.4]{Bo07}, we can find a probability space $(\Omega,\cA)$ and a measure $\kappa$ on $(\Omega,\cA)$, and random variables $X_n, X : \Omega \rightarrow \cX$ such that the random variables $X_n$ and $X$ under the law $\kappa$ have distributions $\nu_{n,0}$ and $\nu_0$ and such that $X_n$ converges to $X$ $\kappa$ almost surely.

		Thus, by assumption, there is a set $B \in \cA$ of $\kappa$ measure $1$ on which $X_n \rightarrow X$ and on which $\mu_n(\cdot \, | \, X_n) = \hat{\mu}(\cdot \, | \, X_n)$ converges to $\mu(\cdot \, | \, X) = \hat{\mu}(\cdot \, | \, X)$. It follows by Proposition \ref{proposition:Gamma_convergence_relative_entropy} that on this set there are measures $\pi_n(\cdot \, | \, X_n)$ such that:
		\begin{gather*}
		\lim_{n \rightarrow \infty} \pi_n(\cdot \, | \, X_n) = \nu(\cdot \, | \, X) \qquad \text{weakly}, \\
		\limsup_n S(\pi_n(\cdot \, | \, X_n) | \, \mu_n(\cdot \, | \, X_n)) \leq S(\nu(\cdot \, | \, X) | \, \mu(\cdot \, | \, X)).
		\end{gather*}
		We could construct a sequence of measures $\nu_n$ out of $\nu_0$ and the conditional kernels $\pi_n$. To establish the $\limsup_n$ inequality for the relative entropies, however, we will need to interchange a $\limsup_n$ and an integral by using Fatou's lemma. At this point, we are not able to give a dominating function that will allow the application of Fatou. To solve this issue, we will use $\pi_n$ only when its relative entropy is not to large.

		Set $A_n := \left\{\omega \in \Omega \, | \, S(\pi_n(\cdot \, | \, X_n(\omega)) \, | \, \mu_n(\cdot \, | \, X_n(\omega))) \leq S(\nu(\cdot \, | \, X(\omega)) \, | \, \mu(\cdot \, | \, X(\omega))) + 1 \right\}$. Note that $\liminf_n A_n$ has $\kappa$ measure $1$. Now set
		\begin{equation*}
		\nu_n(\cdot \, | \, X_n) := \begin{cases}
		\pi_n(\cdot \, | \, X_n) & \text{if } X_n \in A_n, \\
		\mu_n(\cdot \, | \, X_n) & \text{if } X_n \notin A_n,
		\end{cases}
		\end{equation*}
		and define $\nu_n(\dd x, \dd y) = \int \nu_n( \dd y \, | \, x) \nu_0(\dd x)$. We will establish that
		\begin{enumerate}[(1)]
			\item $\limsup_n S(\nu_n \, | \, \mu_n) \leq S(\nu \, | \, \mu)$,
			\item $\nu_n \rightarrow \nu$.
		\end{enumerate}
		
		We start with the proof of (1). By construction and Proposition \ref{proposition:relative_entropy_decomposition}, we have
		\begin{align*}
		\limsup_n S(\nu_n \, | \, \mu_n) & \leq \limsup_n  S_\cX(\nu_{0,n} \, | \, \mu_{0,n}) + \limsup_n \int S(\nu_n(\cdot \, | \, x) \, | \, \mu_n(\cdot \, | \, x)) \nu_{n,0}(\dd x) \\
		& \leq S_\cX(\nu_{0} \, | \, \mu_{0}) + \limsup_n \bE_{\kappa}\left[ S(\nu_n(\cdot \, | \, X_n) \, | \, \mu_n(\cdot \, | \, X_n)) \right] \\
		& \leq S_\cX(\nu_{0} \, | \, \mu_{0}) + \bE_{\kappa}\left[  \limsup_n S(\nu_n(\cdot \, | \, X_n) \, | \, \mu_n(\cdot \, | \, X_n)) \right] \\
		& \leq S_\cX(\nu_{0} \, | \, \mu_{0}) + \bE_{\kappa}\left[ S(\nu(\cdot \, | \, X) \, | \, \mu(\cdot \, | \, X)) \right] \\
		& = S(\nu \, | \, \mu).
		\end{align*}
		In line 3, we used Fatou's lemma, using as an upper bound the function $S(\nu(\cdot \, | \, X) \, | \, \mu(\cdot \, | \, X)) + 1$. This function has finite $\kappa$ integral as
		\begin{equation*}
		\bE_{\kappa} \left[ S(\nu(\cdot \, | \, X) \, | \, \mu(\cdot \, | \, X)) \right] = S(\nu \, | \mu) - S(\nu_0 \, | \, \mu_0)  < \infty.
		\end{equation*}
		Next, we establish (2): $\nu_n \rightarrow \nu$. By (1) and Proposition \ref{proposition:equi_coercivity_relative_entropy} the collection of measures $\nu_n$ is tight. As a consequence, it suffices to establish that $\int h \dd \nu_n \rightarrow \int h \dd \nu$ for a strictly dense set of functions $h$ that is also an algebra by the Stone-Weierstrass theorem for the strict topology. Clearly, the set of linear combinations of functions of the form $h(x,y) = f(x)g(y)$ is an algebra that separates points. Thus, it suffices to establish convergence for $h(x,y) = f(x)g(y)$ only. For $h$ of this form, we have
		\begin{align*}
		\int f(x)g(y) \nu_n(\dd x,\dd y) & = \int f(x) \left( \int g(y) \nu_n(\dd y \, | x) \right) \nu_{n,0}(\dd x) \\
		& = \bE_{\kappa}\left[ f(X_n) \left( \int g(y) \nu_n(\dd y \, | X_n) \right)\right] 
		\end{align*}
		By the weak convergence of $\nu_n(\cdot \, | \, X_n)$ to $\nu(\cdot \, | \, X)$ on a set of $\kappa$ measure $1$, we find by the dominated convergence theorem that
		\begin{equation*}
		\bE_{\kappa}\left[ f(X_n) \left( \int g(y) \nu_n(\dd y \, | X_n) \right)\right] \rightarrow \bE_{\kappa}\left[ f(X) \left( \int g(y) \nu(\dd y \, | X) \right)\right].
		\end{equation*}
		This establishes that $\int h \dd \nu_n \rightarrow \int h \dd \nu$ for $h(x,y) = f(x)g(y)$ and thus that $\nu_n \rightarrow \nu$. 
	\end{proof}

\begin{proof}[Proof of Proposition \ref{proposition:gamma_convergence_integral_entropy}]
	First of all: we can choose finite collections of times $\cT_k := \left\{0 = t_0^k < t_1^k < \dots, < t^k_{i_{\max(k)}} \right\}$,  $k \in \{1,2,\dots\}$ such that: 
	\begin{itemize}
		\item $\cT_k \subseteq \cT_{k+1}$,
		\item $t_{i_{\max}(k)} \geq k$,
		\item For all $k$, and $i \leq i_{\max(k)}$: $t_{i+1}^{k} \leq t_i^k + k^{-1}$,
		\item For all $k$, and $i \leq i_{\max(k)}$: $S_{t_{i+1}^{k}}(\bQ \, | \, \PR_x) \leq S_{t_{i}^{k}}(\bQ \, | \, \PR_x) + k^{-1}$.
	\end{itemize}
	For any $k$, we find by Lemma \ref{lemma:conditional_gamma_entropy} and induction over the finite collection of times in $\cT_k$ that there are measures $\bQ_n^k \in \cP(D_E(\bR^+))$ such that 
	\begin{enumerate}[(1)]
		\item for all  $t \geq t_{i_{\max}(k)}$:
		\begin{multline*}
		\limsup_n S_t(\bQ_n^k \, | \, \PR_{x_n}) = \limsup_n S_{\cT_k}(\bQ_n^k \, | \, \PR_{x_n}) \\
		\leq S_{\cT_k}(\bQ \, | \, \PR_x) \leq S_{t^k_{i_{\max}(k)}}(\bQ \, | \, \PR_x) \leq S_t(\bQ \, | \, \PR_x).
		\end{multline*}
		\item If $t < t_{i_{\max}(k)}$, let $t^k_{i^*}$ be the smallest time in $\cT_k$ such that $t^k_{i^*} \geq t$. Then:
		\begin{multline*}
		\limsup_n S_t(\bQ_n^k \, | \, \PR_{x_n}) \leq \limsup_n S_{t^k_{i^*}}(\bQ_n^k \, | \, \PR_{x_n}) = \limsup_n S_{\cT_k\cap [0,t^k_{i^*}]}(\bQ_n^k \, | \, \PR_{x_n}) \\
		\leq S_{\cT_k \cap [0,t^k_{i^*}]}(\bQ \, | \, \PR_x) \leq S_{t^k_{i^*}}(\bQ \, | \, \PR_x) \leq S_t(\bQ \, | \, \PR_x) + k^{-1}.
		\end{multline*}
	\end{enumerate}
	Thus, we obtain for all $t \geq 0$ that
	\begin{equation*}
	\sup_{n} \sup_k S_t(\bQ_n^k \, | \, \PR_{x_n}) < \infty
	\end{equation*}
	which implies by Proposition \ref{proposition:equi_coercivity_relative_entropy} that the family $\bQ_n^k$ is tight. By construction, i.e. Lemma  \ref{lemma:conditional_gamma_entropy}, the restrictions of the measures $\bQ_n^k$ to the set of times $\cT_k$ converge to the restriction of $\bQ$ to the times in $\cT_k$. A straightforward diagonal argument can be used to find $k(n)$ such that restriction of the measures $\bQ_n := \bQ_{n}^{k(n)}$ to the union $\bigcup_k \cT_k$ to $\bQ$ restricted to the union $\bigcup_k \cT_k$. This however, establishes that $\bQ_n$ converges to $\bQ$ by Theorem 3.7.8 of \cite{EK86}.
	
\end{proof}

\subsection{Regularity of the resolvent in $x$}\label{subsection:resolvent_continuous}

We proceed with the proof of Proposition \ref{proposition:continuity_of_resolvent}: establishing $R(\lambda)h \in C_b(E)$. For the proof of upper semi-continuity of $x \mapsto R(\lambda)h(x)$ we use the following technical result that we state for completeness.

\begin{lemma}[Lemma 17.30 in \cite{AlBo06}] \label{lemma:upper_semi_continuity_abstract}
	Let $\cX$ and $\cY$ be two Polish spaces. Let $\phi : \cX \rightarrow \cK(\cY)$, where $\cK(\cY)$ is the space of non-empty compact subsets of $\cY$, be upper hemi-continuous. That is: if $x_n \rightarrow x$ and $y_n \rightarrow y$ and $y_n \in \phi(x_n)$, then $y \in \phi(x)$. 
	
	Let $f : \text{Graph} (\phi) \rightarrow \bR$ be upper semi-continuous. Then the map $m(x) = \sup_{y \in \phi(x)} f(x,y)$ is upper semi-continuous.
\end{lemma}

\begin{proof}[Proof of Proposition \ref{proposition:continuity_of_resolvent}]
	Fix $\lambda > 0$ and $h \in C_b(E)$. Denote as before
	\begin{equation*}
	\cS_\lambda(\bQ \, | \, \PR_x) := \int_0^\infty S_t(\bQ \, | \, \PR_x) \tau_\lambda(\dd t),
	\end{equation*} 
	to shorten notation. By Lemma \ref{lemma:compact_level_sets_integrated_relative_entropy} the map $\bQ \mapsto \cS_\lambda(\bQ \, | \, \PR_x)$ has compact sub-levelsets and is lower semi-continuous. As $h$ is bounded we have
	\begin{equation*}
	R(\lambda)h(x) = \sup_{\bQ \in \Gamma_x} \left\{ \int h(X(t)) \bQ(\dd X)\tau_\beta(\dd t) - \cS_\lambda(\bQ \, | \, \PR_x)  \right\}
	\end{equation*}
	where $\Gamma_x := \left\{\bQ \in \cP \, \middle| \, \cS_\lambda(\bQ \, | \, \PR_x) \leq 2\vn{h} \right\}$. Note that $\Gamma_x$ is non-empty and compact.
	
	Due to the lower semi-continuity of $\cS_\lambda$ and the continuity of the integral over $h$, it follows that $x \mapsto R(\lambda)h(x)$ is upper semi-continuous by Lemma 17.30 of \cite{AlBo06} if the collection of sets $\Gamma_x$ is \textit{upper hemi-continuous}; or in other words: if $\bQ_n \in \Gamma_{x_n}$ and $(x_n,\bQ_n) \rightarrow (x,\bQ)$ then $\bQ \in \Gamma_x$. This, however, follows directly from the lower semi-continuity of $\cS_\lambda$.
	
	\smallskip
	
	Next, we establish lower semi-continuity of $x \mapsto R(\lambda)h(x)$. Let $x_n$ be a sequence converging to $x$. Pick $\bQ$ so that
	\begin{equation*}
	R(\lambda)h(x) = \int_0^\infty \int h(X(t)) \bQ(\dd X) \tau_\lambda(\dd t) - \cS_\lambda(\bQ \, | \PR_x) 
	\end{equation*}
	It follows by Proposition \ref{proposition:gamma_convergence_integral_entropy} that there are $\bQ_n \in \cP(D_E(\bR^+))$ such that $\bQ_n \rightarrow \bQ$ and $\limsup_n \cS_\lambda(\bQ_n \, | \, \PR_{x_n}) \leq \cS_\lambda(\bQ \, | \, \PR_x)$. We obtain that
	\begin{align*}
	\liminf_n R(\lambda)h(x_n) & \geq \liminf_n \int_0^\infty \int h(X(t)) \bQ_n(\dd X) \tau_\lambda(\dd t) - \cS_\lambda(\bQ_n \, | \, \PR_{x_n}) \\
	& \geq \int_0^\infty \int h(X(t)) \bQ(\dd X) \tau_\lambda(\dd t) - \cS_\lambda(\bQ \, | \, \PR_x)  = R(\lambda)h(x)
	\end{align*}
	establishing lower semi-continuity.
\end{proof}

\subsection{Regularity of the resolvent in $h$} \label{subsection:Markov_continuity}

We proceed with establishing that the resolvent is sequentially strictly continuous in $h$, uniformly for small $\lambda$.

\begin{lemma} \label{lemma:strict_equicontinuity_of_resolvent}
	For every $\lambda_0>0$ the family of maps $\{R(\lambda)\}_{0<\lambda \leq \lambda_0}$ is sequentially strictly equi-continuous. That is: for every $h_1,h_2 \in C_b(E)$, every compact set $K \subseteq E$ and $\delta > 0$ there is a compact set $\hat{K} \subseteq E$ such that
	\begin{equation*}
	\sup_{x \in K} R(\lambda) h_1(x) - R(\lambda)h_2(x) \leq \delta + \sup_{y \in \hat{K}} h_1(y) - h_2(y).
	\end{equation*}
	for all $0 < \lambda \leq \lambda_0$.
\end{lemma}

As above denote by $\cS_{\lambda}(\bQ \, | \, \PR) := \int_0^\infty S_t(\bQ \, | \, \PR) \tau_{\lambda}(\dd t)$.

\begin{proof}
	Fix $h_1,h_2 \in C_b(E)$, $\lambda_0 > 0$, $\delta > 0$ and a compact set $K \subseteq E$.
	
	Pick an arbitrary $\lambda$ such that $0 < \lambda \leq \lambda_0$.	For $x \in K$, let $\bQ_{x,\lambda} \in \cP$ be the measure such that
	\begin{equation*}
	R(\lambda) h_1(x) = \left\{ \int_0^\infty \int h_1(X(t)) \bQ_{x,\lambda}(\dd X) \tau_\lambda(\dd t) - \cS_\lambda(\bQ_{x,\lambda} \, | \, \PR_x ) \right\}
	\end{equation*}
	and note that $\cS_\lambda(\bQ_{x,\lambda} \, | \, \PR_x) \leq 2 \vn{h_1}$.	It follows that
	\begin{align*} 
	& \sup_{x \in K} R(\lambda) h_1(x) - R(\lambda)h_2(x) \\
	& \leq \sup_{\bQ_1 \in \cP} \left\{ \int_0^\infty \int h_1(X(t)) \bQ_1(\dd X) \tau_{\lambda}(\dd t) - \cS_\lambda(\bQ_{1} \, | \, \PR_x )  \right\} \\
	& \qquad - \sup_{\bQ_2 \in \cP} \left\{ \int_0^\infty \int h_2(X(t)) \bQ_2(\dd X) \tau_{\lambda}(\dd t) - \cS_\lambda(\bQ_{2} \, | \, \PR_x )  \right\}. \\
	& \leq \sup_{x \in K} \int_0^\infty \int h_1(X(t)) - h_2(X(t)) \bQ_{x,\lambda}(\dd X) \tau_\lambda(\dd t).
	\end{align*}
	Denote by $T(\lambda) := - \lambda \log \frac{\delta}{2\vn{h_1-h_2}}$. Then it follows that
	\begin{equation*}
	\sup_{x \in K} R(\lambda) h_1(x) - R(\lambda)h_2(x) \leq \frac{\delta}{2} + \sup_{x \in K} \int_0^{T(\lambda)} \int h_1(X(t)) - h_2(X(t)) \bQ_{x,\lambda}(\dd X) \tau_\lambda(\dd t).
	\end{equation*}
	Now denote by $\widehat{\bQ}_{x,\lambda}$ the measure that equals $\bQ_{x,\lambda}$ on the time interval $[0,T(\lambda)]$ and satisfies $S_{T(\lambda)}(\widehat{\bQ}_{x,\lambda} \, | \PR_x) = S(\widehat{\bQ}_{x,\lambda} \, | \PR_x)$. By Proposition \ref{proposition:uniform_compact_levelsets} the set of the measures $\widehat{\bQ}_{x,\lambda}$, $x \in K, 0 < \lambda \leq \lambda_0$, is relatively compact, which implies we can find a $\widehat{K} \subseteq E$ with probability $(1-\frac{\delta}{2})$ the trajectories stay in $\widehat{K}$. We conclude that
	\begin{equation*}
	\sup_{x \in K} R(\lambda) h_1(x) - R(\lambda)h_2(x) \leq \delta + \sup_{y \in \hat{K}} h_1(y) - h_2(y)
	\end{equation*}
	for all $\lambda$ such that $0 < \lambda \leq \lambda_0$.
\end{proof}

\subsection{Strong continuity of the resolvent and semigroup} \label{subsection:strong_continuity}

We establish that as $\lambda \downarrow 0$ the resolvents converge to the identity operator. We also establish strict continuity of the semigroup.

\begin{lemma} \label{lemma:continuity_resolvent_at_0}
	For $h \in C_b(E)$  we have $\lim_{\lambda \rightarrow 0} R(\lambda) h = h$ for the strict topology.
\end{lemma}

\begin{proof}
	As $\vn{R(\lambda)h} \leq \vn{h}$ strict convergence $\lim_{\lambda \rightarrow 0} R(\lambda) h = h$ follows by proving uniform convergence on compact sets $K \subseteq E$.
	
	If we choose for $\bQ$ the measure $\PR_{x}$ in the defining supremum of $R(\lambda)h(x)$, we obtain the upper bound
	\begin{equation*}
	R(\lambda) h (x) - h(x) \geq \int_0^\infty \int h(X(t)) - h(x) \PR_x(\dd X) \tau_\lambda(\dd t).
	\end{equation*}
	As the measures $\{\PR_x\}_{x \in K}$ are tight, we have control on the modulus of continuity of the trajectories $t \mapsto X(t)$. This implies that the right-hand side converges to $0$ as $\lambda \downarrow 0$ uniformly for $x \in K$.
	
	\smallskip
	
	We prove the second inequality. Fix $\varepsilon \in (0,4\vn{h})$, we prove that for $\lambda$ sufficiently small, we have $\sup_{x \in K}  R(\lambda) h (x) - h(x) \leq \varepsilon$. First of all, let $T(\lambda) := - \lambda \log \left(\frac{\varepsilon}{4 \vn{h}}\right)$ and let $\bQ_{x,\lambda}$ optimize $R(\lambda) h(x)$. We then have
	\begin{equation} \label{eqn:upper_bound_R_lambda_to_0}
	R(\lambda)h(x) - h(x) \leq \frac{1}{2}\varepsilon + \int_0^{T(\lambda)} \int h(X(t)) - h(x) \, \bQ_{x,\lambda}(\dd X)  \tau_\lambda(\dd t).
	\end{equation}
	Also note that as in Lemma \ref{lemma:strict_equicontinuity_of_resolvent} we have $\cS_\lambda(\bQ_{x,\lambda} \, | \, \PR_x) \leq 2\vn{h}$. This implies, using that $t \mapsto S_t$ is increasing in $t$, that
	\begin{equation}
	S_{T(\lambda)}(\bQ_{x,\lambda} \, | \, \PR_x) \leq 8 \vn{h}^2 \varepsilon^{-1}.
	\end{equation}
	Denote by $\widehat{\bQ}_{x,\lambda}$ the measures that equal $\bQ_{x,\lambda}$ up to time $T(\lambda)$ and satisfy $S_{T(\lambda)}(\widehat{\bQ}_{s,\lambda} \, | \, \PR_x) = S(\widehat{\bQ}_{s,\lambda} \, | \, \PR_x)$. 
	
	Now let $\lambda \leq \lambda^* := \left(\log 4\vn{h}\varepsilon^{-1} \right)^{-1}$. Then $T(\lambda) \leq 1$ and we obtain for all $s \geq 1$ that
	\begin{equation}
	S_{s}(\widehat{\bQ}_{x,\lambda} \, | \, \PR_x) \leq 8 \vn{h}^2 \varepsilon^{-1}.
	\end{equation}
	By Proposition \ref{proposition:equi_coercivity_relative_entropy}, the measures $\{\widehat{\bQ}_{x,\lambda}\}_{0< \lambda \leq \lambda^*,x \in K}$ form a tight family. Replacing $\bQ_{x,\lambda}$ by $\widehat{\bQ}_{x,\lambda}$ in \eqref{eqn:upper_bound_R_lambda_to_0}, using the tightness of the family of measures, the upper bound follows as for the lower bound.
\end{proof}

\begin{lemma} \label{lemma:continuity_semigroup}
	For each $h \in C_b(E)$ we have that $t \mapsto V(t)h$ is continuous for the strict topology.
\end{lemma}

\begin{proof}
	The map $t \mapsto S(t)e^f$ is strictly continuous by Theorem 3.1 of \cite{Kr19d} and bounded away from $0$. Thus a straightforward verification shows that also $V(t)f = \log S(t)e^f$ is strictly continuous.

	%
	%
	%
	
\end{proof}

\subsection{Measurability of the optimal measure}

In Section \ref{section:proofs_main_results} below, we will apply the resolvent to the resolvent. This means we have to perform an optimization procedure twice. In particular, this implies we have to integrate over the outcome of the first supremum. To treat this procedure effectively, we need measurability of the optimizing measure.

\begin{lemma} \label{lemma:measurable_selection_in_concatenation}
	Let $h \in C_b(E)$ and $\lambda > 0$. There exists a measurable map $x \mapsto \bQ_{x}$ such that $\bQ_x \in \cP$ and
	\begin{equation*}
	R(\lambda)h(x) = \int_0^\infty \left[ \int h(Y(t)) \, \bQ_x(\dd Y) - S_t(\bQ_x \, | \, \PR_x ) \right] \tau_\lambda(\dd t).
	\end{equation*}
\end{lemma}

We base the proof of this result on a measurable-selection theorem. We state it for completeness.

\begin{theorem}[Theorem 6.9.6 in \cite{Bo07}] \label{theorem:measurable_selection}
	Let $X,Y$ be Polish spaces and let $\Gamma$ be a measurable subset of $X \times Y$. Suppose that the set $\Gamma_x := \left\{y \, \middle| \, (x,y) \in \Gamma  \right\}$ is non-empty and $\sigma$-compact for all $x \in X$. Then $\Gamma$ contains the graph of a Borel measurable mapping $f : X \rightarrow Y$.
\end{theorem}

We will apply this result below by using the following argument. Let $f,g$ be measurable maps $f,g : X \times Y \rightarrow (-\infty,\infty]$. The set $\{(x,y) \, | \, f(x,y) = g(x,y)\}$ is measurable as it equals $\{ (x,y) \, | \, f(x) - g(x) = 0\}$ which is the inverse image of $\{0\}$ and hence measurable.

\begin{proof} [Proof of Lemma \ref{lemma:measurable_selection_in_concatenation}]
	We aim to apply Theorem \ref{theorem:measurable_selection}. Thus, we have to establish that the set $\Gamma \subseteq E \times \cP$ defined by
	\begin{equation*}
	\Gamma := \left\{(x, \bQ)  \, \middle| \,  R(\lambda)h(x) = \int_0^\infty \left[\int h(Y(t)) \bQ(\dd Y) - S_t(\bQ \, | \, \PR_x) \right] \tau_{\lambda}(\dd t) \right\}
	\end{equation*}
	is measurable and that $\Gamma_x := \left\{\bQ \, \middle| \, (x,\bQ) \in \Gamma \right\}$ is non-empty and $\sigma$-compact. 
	
	\smallskip

	Similarly as in the proof of Proposition \ref{proposition:continuity_of_resolvent}, we find that $\Gamma_x$ is compact and non-empty. We also saw in that proof that the map $(x,\bQ) \mapsto \int h(Y(t)) \bQ(\dd Y) - S(\bQ \, | \, \PR_x ) \, \tau_{\lambda}(\dd t)$ is upper semi-continuous. As $x \mapsto R(\lambda)h(x)$ is continuous by Proposition \ref{proposition:continuity_of_resolvent} we see that the set $\Gamma$ is the set of points where two measurable functions agree implying that $\Gamma$ is measurable.  An application of Theorem \ref{theorem:measurable_selection} concludes the proof.	
\end{proof}

\section{Proofs of the main results} \label{section:proofs_main_results}

In this section, we prove the two main results: Theorem \ref{theorem:Markov_solving_HJ} and Proposition \ref{proposition:resolvents_approximate_semigroup}. We argued in Section \ref{section:discussion_proofs} that the first result follows by establishing that $R(\lambda)$ is a classical left-inverse of $(\bONE - \lambda H)$ and that the family $R(\lambda)$ is a pseudo-resolvent. We establish these two properties in Sections \ref{subsection:Markov_invert} and \ref{subsection:Markov_pseudoresolvent}. The proof of Proposition \ref{proposition:resolvents_approximate_semigroup} is carried out in Section \ref{subsection_resolvent_to_semigroup}.

\subsection{$R(\lambda)$ is a classical left-inverse of $\bONE- \lambda H$} \label{subsection:Markov_invert}

The proof that $R(\lambda)$ is a classical left-inverse of $\bONE- \lambda H$ is based on a well known integration by parts formula for the exponential distribution for bounded measurable functions $z$ on $\bR^+$ we have 
\begin{equation} \label{eqn:integral_transform}
\lambda \int_0^\infty  z(t) \, \tau_\lambda( \dd t) = \int_0^\infty \int_0^t z(s) \, \dd s \, \tau_\lambda(\dd t).
\end{equation}

A generalization is given by the following lemma.

\begin{lemma} \label{lemma:integral_transform_complex_measures}
	Fix $\lambda > 0$ and $\bQ \in \cP(D_E(\bR^+))$. Let $z$ be a measurable function on $E$. Then we have
	\begin{equation*}
	\lambda \int_0^\infty \int  z(X(t)) \, \bQ(\dd X)  \, \tau_\lambda(\dd t) = \int_0^\infty \int \int_0^t z(X(s)) \, \dd s \, \bQ(\dd X) \tau_\lambda(\dd t).
	\end{equation*}
\end{lemma}

%
%
%
%
%
%

The lemma allows us to rewrite the application of $R(\lambda)$ to $f - \lambda g$ in integral form. The integral that comes out can be analyzed using the definition of $H$ in terms of exponential martingales. This leads to the desired result.

\begin{proposition} \label{proposition:Resolvent_on_Hamiltonian_Markov}
	Let $\PR_x$ be a collection of Markov measures as in Condition \ref{condition:Markov_single}. For all $\lambda > 0$, $x \in E$ and $(f,g) \in H$, we have $R(\lambda)(f - \lambda g)(x) = f(x)$.
\end{proposition}

\begin{proof}
		Fix $\lambda > 0$, $x \in E$ and $(f,g) \in H$. We start by proving $R(\lambda)(f -\lambda g)(x) \leq f(x)$. Set $h = f - \lambda g$. By Lemma \ref{lemma:integral_transform_complex_measures} we have
		\begin{align*}
		R(\lambda)h (x) & = \sup_{\bQ \in \cP} \left\{ \int_0^\infty \tau_{\lambda}(\dd t) \left[ \int \left( f(X(t)) - \lambda g(X(t))  \right) \bQ(\dd X) - S_t(\bQ \, | \, \PR_x) \right] \right\} \\
		& = \sup_{\bQ \in \cP} \left\{ \int_0^\infty \tau_{\lambda}(\dd t) \left[ \int \left( f(X(t)) - \int_0^t g(X(s)) \dd s \right) \bQ(\dd X) - S_t(\bQ \, | \, \PR_x) \right] \right\}.
		\end{align*}
		By optimizing the integrand, we find by Lemma \ref{lemma:DV_variational}
		\begin{align*}
		& R(\lambda)(f - \lambda g)(x) \\
		& \leq \int_0^\infty \tau_\lambda(\dd t) \left\{ \sup_{\bQ \in \cP} \left[ \int \bQ(\dd X) \left( f(X(t)) - \int_0^t g(X(s)) \dd s\right) - S_t(\bQ \, | \, \PR_x) \right] \right\} \\
		& = \int_0^\infty \tau_\lambda(\dd t) \left\{ \log \bE\left[e^{f(X(t)) - \int_0^t g(X(s)) \dd s} \, \middle| \, X(0) = x \right]   \right\}.
		\end{align*} 
		As $(f,g) \in H$ we can reduce the inner expectation to time $0$ by using the martingale property. This yields
		\begin{equation*}
		R(\lambda)(f - \lambda g)(x) \leq \int \tau_\lambda(\dd t) \left\{ \log \bE\left[e^{f(X(0))} \, \middle| \, X(0) = x \right]   \right\} = f(x),
		\end{equation*}	
		establishing the first inequality.		
		
		We now prove the reverse inequality $R(\lambda)(f -\lambda g)(x) \geq f(x)$. To do so, we construct a measure $\bQ$ that achieves the supremum. For each time $t \geq 0$, define the measure $\bQ^t$ via the Radon-Nykodim derivative
	\begin{equation*}
	\frac{\dd \bQ^{t}}{\dd \PR} (X) = \exp\left\{f(X(t)) - f(X(0)) - \int_0^t g(X(s)) \dd s \right\}.
	\end{equation*}
	Note that as $t \mapsto \exp\left\{f(X(t)) - f(X(0)) - \int_0^t g(X(s)) \dd s \right\}$ is a $\PR_x$ martingale, we have for $s \leq t$ that $\bQ^t |_{\cF_s} = \bQ^s |_{\cF_s}$. Thus, standard arguments show that there is a measure $\bQ \in \cP$ such that $\bQ |_{\cF_t} = \bQ^t |_{\cF_t}$. Note that by construction, we have $\bQ(X(0) = x) = 1$. Using this measure $\bQ$, applying Lemma \ref{lemma:integral_transform_complex_measures}, we obtain
	\begin{align*}
	& R(\lambda)\left(f - \lambda H f\right)(x) \\
	& \geq \int_0^\infty \int  \left[\left( f(X(t)) - \int_0^t g(X(s)) \dd s \right) \right. \\
	& \hspace{4cm} \left. - \left(f(X(t)) - f(X(0)) - \int_0^t g(X(s)) \dd s\right)\right] \bQ(\dd X)  \tau_\lambda(\dd t) \\
	& = f(x)
	\end{align*}
	establishing the second inequality.
\end{proof}

\subsection{$R$ is a pseudo-resolvent} \label{subsection:Markov_pseudoresolvent}

The next step is the verification that the family of operators $R(\lambda)$ is a pseudo-resolvent. As in the previous section, this property is essentially an extension of a key property of the exponential distribution. We state it as a lemma that can be verified using basic calculus.

\begin{lemma} \label{lemma:double_exponential_integral_simplification}
	Let $z : \bR^+ \rightarrow \bR$ be a bounded and measurable function. Let $0 < \alpha < \beta$. Then
	\begin{equation*}
	\int_0^\infty z(s) \tau_\beta(\dd s) = \frac{\alpha}{\beta} \int_0^\infty z(s) \tau_\alpha(\dd s) + \left(1 - \frac{\alpha}{\beta}\right) \int_0^\infty \int_0^\infty z(s+u) \, \tau_\beta(\dd u) \, \tau_\alpha(\dd s).
	\end{equation*}
\end{lemma}


Lifting this property to the family $R(\lambda)$ yields the pseudo-resolvent property.

\begin{proposition} \label{proposition:pseudo_resolvent_Markov_entropy}
	For all $h \in C_b(E)$, $x \in E$, and $0 < \alpha < \beta$, we have
	\begin{equation} \label{eqn:pseudo_resolvent_Markov}
	R(\beta)h(x) = R(\alpha) \left(R(\beta)h - \alpha \frac{R(\beta)h - h}{\beta} \right)(x). 
	\end{equation}
\end{proposition}

Note that the right-hand side of \eqref{eqn:pseudo_resolvent_Markov} can be rewritten as
\begin{equation} \label{eqn:explicit_double_resolvent}
\begin{aligned}
& R(\alpha) \left(R(\beta)h - \alpha \frac{R(\beta)h - h}{\beta} \right)(x) \\
& = \sup_{\bQ \in \cP} \left\{\int_0^\infty \int \left(1-\frac{\alpha}{\beta}\right)R(\beta)h(X(t)) - \frac{\alpha}{\beta}h(X(t)) \, \bQ(\dd X) \,  -S_t(\bQ \, | \, \PR_{x} )  \tau_\alpha(\dd t) \right\} \\
&= \sup_{\bQ \in \cP}\left\{\int  \left(1-\frac{\alpha}{\beta}\right) \sup_{\bQ^{t} \in \cP} \left[\int_0^\infty \int  h(Y(s)) \bQ^{t}(\dd Y) - S_s(\bQ_{2,t} \, | \, \PR_{X(t)} ) \tau_\beta(\dd s)\right]  \bQ(\dd X) \, \tau_\alpha(\dd t) \right. \\
& \left. \hspace{5cm} + \int_0^\infty  \int \frac{\alpha}{\beta}h(X(t)) \, \bQ(\dd X)  -S_t(\bQ \, | \, \PR_{x} ) \tau_\alpha(\dd t) \right\}. 
\end{aligned}
\end{equation}
To establish \eqref{eqn:pseudo_resolvent_Markov} we establish two inequalities. To do so, we will consider two techniques. First, to prove that the right-hand side is dominated by the left-hand side, we need to concatenate optimizers. To establish the other inequality, we will take an optimizer for $R(\beta)h$ and make a time-dependent splitting, so that we can dominate the first part in the first optimization, and the second by the second optimization in \eqref{eqn:explicit_double_resolvent}. The proof of Proposition \ref{proposition:pseudo_resolvent_Markov_entropy} will be carried out in the next two sections. Both proofs are inspired by the proof of Lemma 8.20 of \cite{FK06} where the pseudo-resolvent property is established for the deterministic case. 

\smallskip

\subsubsection{Concatenating measures} \label{section:subsub_concationation}

In this section, we will prove that 
\begin{equation} \label{eqn:pseudo_resolvent_Markov_concatenate}
R(\beta)h \geq R(\alpha) \left(R(\beta)h - \alpha \frac{R(\beta)h - h}{\beta} \right).
\end{equation}
We start by introducing the procedure of concatenating measures. For $s \geq 0$ and $X,Y \in D_E(\bR^+)$ such that $X(s) = Y(0)$, define the concatenation $\kappa^s_{X,Y} \in D_E(\bR^+)$ by
\begin{equation*}
\kappa^s_{X,Y}(t) = \begin{cases}
X(t) & \text{if } t \leq s, \\
Y(t-s) & \text{if } t \geq s.
\end{cases}
\end{equation*}

For $\bQ \in \cP(D_E(\bR^+))$ and map $q: D_E(\bR^+) \rightarrow \cP(D_E(\bR^+))$ with $q(X) = \bQ^{X}$ that is $\cF_s$ measurable and supported on a set such that $Y(0) = X(s)$ define the measure
\begin{equation} \label{eqn:concatenation_measure}
\bQ \odot_s q (\dd Z) = \int \int \bQ(\dd X) \bQ^{X}(\dd Y) \delta_{\kappa^s_{X,Y}}(\dd Z)
\end{equation}

Before starting with the proof of \eqref{eqn:pseudo_resolvent_Markov_concatenate}, we start with the computation of the relative entropy of $\bQ \odot_s q$.

\begin{lemma} \label{lemma:decomposition_of_entropy_in_proof_concatenation}
	Fix $s$, $t > s$ and $X \in D_E(\bR^+)$. We have
	\begin{equation*}
	S_t(\bQ \odot_s q \, | \, \PR_x) = S_s(\bQ \, | \, \PR_x) + \int S_{t-s}(\bQ^{X} \, | \, \PR_{X(s)}) \bQ(\dd X).
	\end{equation*}
\end{lemma}

\begin{proof}
	Fix $s$, $t > s$ and $X \in D_E(\bR^+)$. Define the measure $\widehat{\bQ}^{s,X} (\dd Z) = \int \bQ^{X}(\dd Y) \delta_{\kappa^s_{X,Y}}(\dd Z)$. It follows by definition that $\bQ \odot_s q(\dd Z) = \int \bQ(\dd X)\widehat{\bQ}^{s,X} (\dd Z)$ and that $\bQ^{s,X}$ is the regular conditional measure of $\bQ \odot_s q$ conditioned on $\cF_s$. Denote by $\PR_{[0,s],X}$ the measure $\PR_x$ conditioned on $\cF_s$.
	
	Proposition \ref{proposition:relative_entropy_decomposition}, applied for the conditioning on $\cF_s$ yields
	\begin{equation*}
	S_t(\bQ \odot_s q \, | \, \PR_x) = S_s(\bQ \, | \, \PR_x) + \int S_{t}(\widehat{\bQ}^{s,X} \, | \, \PR_{[0,s],X}) \bQ(\dd X).
	\end{equation*}
	 Both measures $\widehat{\bQ}^{s,X}$ and $\PR_{[0,s],X}$ are supported by trajectories that equal $X$ on the time interval $[0,s]$. Shifting both measures by $s$, we find $\bQ^X$ (as defined above) and by the Markov property $\PR_{X(s)}$. As this shift is a isomorphism of measure spaces, we find
	 \begin{equation*}
	 S_t(\bQ \odot_s q \, | \, \PR_x) = S_s(\bQ \, | \, \PR_x) + \int S_{t-s}(\bQ^{X} \, | \, \PR_{X(s)}) \bQ(\dd X),
	 \end{equation*}
	 which establishes the claim.
\end{proof}

%

\begin{proof}[Proof of \eqref{eqn:pseudo_resolvent_Markov_concatenate}]
	Fix $h \in C_b(E)$, $x \in E$ and $0 < \alpha < \beta$. 
	
	We aim to establish \eqref{eqn:pseudo_resolvent_Markov_concatenate} by taking the optimizers for both optimization procedures on the right-hand side and to concatenate them. This will yield a new measure that also turns up in the optimization procedure on the left-hand side, thus establishing the claim. For the concatenation, we use Lemma \ref{lemma:decomposition_of_entropy_in_proof_concatenation} to put together the relative entropies of both procedures finish with Lemma \ref{lemma:double_exponential_integral_simplification} to obtain the correct integral form.
	
	\smallskip
	
	Thus, let $\bQ \in \cP$ be the optimizer of
	\begin{equation*}
	\sup_{\bQ \in \cP} \left\{\int_0^\infty \int \left(1-\frac{\alpha}{\beta}\right)R(\beta)h(X(t)) - \frac{\alpha}{\beta}h(X(t)) \bQ(\dd X) -S_t(\bQ \, | \, \PR_{x} ) \tau_\alpha(\dd t) \right\}.
	\end{equation*}
	For any $y \in E$ let $\bQ^{y} \in \cP$ be the optimizer of
	\begin{equation*}
	\sup_{\bQ' \in \cP} \left\{ \int h(Y(s)) \bQ'(\dd Y) - S_s(\bQ' \, | \, \PR_y ) \, \tau_\beta(\dd s) \right\}.
	\end{equation*}
	Fix $s \geq 0$. We established in Lemma \ref{lemma:measurable_selection_in_concatenation} that the map $q$ defined by $q(y) := \bQ^{y}$ is measurable. Thus, using $\bQ$ and $q$, we define $\bQ_s := \bQ \odot_s q $ as in \eqref{eqn:concatenation_measure}. By definition of $R(\beta)h(x)$, we find 
	\begin{equation} \label{eqn:lower_bound_R1}
	R(\beta) h(x) \geq \int_0^\infty \int h(X(t)) \bQ_s(\dd X) - S_t(\bQ_s \, | \, \PR_{x} ) \, \tau_\beta(\dd t). 
	\end{equation}
	We treat both terms on the right-hand side separately.
	\begin{align*}
	& \int_0^\infty \int h(X(t)) \bQ_s(\dd X) \tau_\beta(\dd t) \\
	& = \int_0^s \int h(X(t)) \bQ_s(\dd X) \tau_\beta(\dd t) + \int_s^\infty \int h(X(t)) \bQ_s(\dd X) \tau_\beta(\dd t) \\
	& = \int_0^s \int h(X(t)) \bQ(\dd X) \tau_\beta(\dd t) + \int_s^\infty \int h(Y(t-s)) \bQ^{X(s)}(\dd Y) \tau_\beta(\dd t) \\
	& = \int_0^s \int h(X(t)) \bQ(\dd X) \tau_\beta(\dd t) + e^{-\beta^{-1} s} \int_0^\infty \int h(Y(t)) \bQ^{X(s)}(\dd Y) \tau_\beta(\dd t).
	\end{align*}
	Using Lemma \ref{lemma:decomposition_of_entropy_in_proof_concatenation} in line 3 below, we find that the second term equals
	\begin{align*}
	& \int_0^\infty S_t(\bQ_s \, | \, \PR_{x} ) \, \tau_\beta(\dd t) \\
	& = \int_0^s S_t(\bQ_s \, | \, \PR_x) \, \tau_\beta(\dd t) + \int_s^\infty S_t(\bQ_s \, | \, \PR_x) \, \tau_\beta(\dd t) \\
	& = \int_0^s S_t(\bQ \, | \, \PR_x) \, \tau_\beta(\dd t) + e^{-\beta^{-1} s} S_s(\bQ \, | \, \PR_x) + \int_s^\infty \int S_{t-s}(\bQ^{X(s)} \, | \, \PR_{X(s)}) \bQ(\dd X) \, \tau_\beta(\dd t) \\
	& = \int_0^s S_t(\bQ \, | \, \PR_x) \, \tau_\beta(\dd t) + e^{-\beta^{-1} s} S_s(\bQ \, | \, \PR_x) + e^{-\beta^{-1} s} \int_0^\infty \int S_{t}(\bQ^{X(s)} \, | \, \PR_{X(s)}) \bQ(\dd X) \, \tau_\beta(\dd t).
	\end{align*}

	Thus, for each fixed $s$, we find a lower bound for $R(\beta)h(x)$. If we multiply this inequality by the probability density $\frac{\beta - \alpha}{\alpha \beta} e^{\beta^{-1}s - \alpha^{-1}s}$ on $\bR^+$ and integrate over $s$, we find
	\begin{align*}
	& R(\beta) h(x) \\
	& \geq \frac{\beta - \alpha}{\alpha \beta} \int_0^\infty e^{\beta^{-1}s - \alpha^{-1} s}\int_0^\infty \int h(X(t)) \, \bQ_s(\dd X) - S_t(\bQ_s \, | \, \PR_{x} ) \, \tau_\beta(\dd t) \, \dd s \\
	& \geq \frac{\beta - \alpha}{\alpha \beta} \int_0^\infty e^{\beta^{-1}s - \alpha^{-1} s} \int_0^s  \int h(X(t)) \, \bQ(\dd X)  \,\tau_\beta(\dd t) \, \dd s\\
	& \qquad +  \frac{\beta - \alpha}{\alpha \beta} \int_0^\infty e^{\beta^{-1}s - \alpha^{-1} s} e^{-\beta^{-1} s} \int_0^\infty \int h(Y(t)) \, \bQ^{X(s)}(\dd Y) \, \tau_\beta(\dd t) \, \dd s \\
	& \qquad - \frac{\beta - \alpha}{\alpha \beta} \int_0^\infty e^{\beta^{-1}s - \alpha^{-1} s}\int_0^s S_t(\bQ \, | \, \PR_x) \, \tau_\beta(\dd t) \, \dd s \\
	& \qquad - \frac{\beta - \alpha}{\alpha \beta} \int_0^\infty e^{\beta^{-1}s - \alpha^{-1} s} e^{-\beta^{-1} s} S_s(\bQ \, | \, \PR_x) \, \dd s \\
	& \qquad - \frac{\beta - \alpha}{\alpha \beta} \int_0^\infty e^{\beta^{-1}s - \alpha^{-1} s} e^{-\beta^{-1} s} \int_0^\infty \int S_{t}(\bQ^{X(s)} \, | \, \PR_{X(s)}) \, \bQ(\dd X) \, \tau_\beta(\dd t) \, \dd s.
	\end{align*}
	The integrals of terms in line three, five and size immediately simplify to integration over $\tau_\alpha(\dd s)$ and $\tau_\beta(\dd t)$. The two other integrals can be simplified by using that for nice functions $G$ we have
	\begin{equation*}
	\frac{\beta - \alpha}{\alpha \beta} \int_0^\infty e^{\beta^{-1}s - \alpha^{-1} s}\int_0^\infty G(t) \, \tau_\beta(\dd t) \, \dd s =	\frac{\alpha}{\beta} \int_0^\infty G(t) \, \tau_\alpha(\dd t).
	\end{equation*}
	Plugging in also the equality $\frac{\beta - \alpha}{\beta} = 1 - \frac{\alpha}{\beta}$, we obtain
	\begin{align*}
	R(\beta) h(x)	& \geq \frac{\alpha}{\beta} \int_0^\infty  \int h(X(t)) \, \bQ(\dd X) \, \tau_\alpha(\dd t) \\
	& \qquad +  \left(1 - \frac{\alpha}{\beta}\right) \int_0^\infty  \int_0^\infty \int h(Y(t)) \, \bQ^{X(s)}(\dd Y) \, \tau_\beta(\dd t) \, \tau_\alpha(\dd s) \\
	& \qquad - \frac{\alpha}{\beta} \int_0^\infty  S_t(\bQ \, | \, \PR_x) \, \tau_\alpha(\dd t) \\
	& \qquad - \left(1 - \frac{\alpha}{\beta}\right) \int_0^\infty   S_s(\bQ \, | \, \PR_x) \, \tau_\alpha( \dd s) \\
	& \qquad - \left(1 - \frac{\alpha}{\beta}\right) \int_0^\infty  \int_0^\infty \int S_{t}(\bQ^{X(s)} \, | \, \PR_{X(s)}) \bQ(\dd X) \, \tau_\beta(\dd t) \, \tau_\alpha(\dd s).
	\end{align*}
	Note that the terms in the third and fourth line together give $-\int_0^\infty  S_t(\bQ \, | \, \PR_x) \, \tau_\alpha(\dd t)$. Changing the roles of $s$ and $t$ in the double integrals, we arrive at the inequality
	\begin{align*}
	R(\beta) h(x)	& \geq \frac{\alpha}{\beta} \int_0^\infty  \int h(X(t)) \bQ(\dd X) \tau_\alpha(\dd t) \\
	& \qquad +  \left(1 - \frac{\alpha}{\beta}\right) \int_0^\infty  \int_0^\infty \int h(Y(s)) \bQ^{X(t)}(\dd Y) \, \tau_\beta(\dd s) \, \tau_\alpha(\dd t) \\
	& \qquad - \int_0^\infty  S_t(\bQ \, | \, \PR_x) \, \tau_\alpha(\dd t) \\
	& \qquad - \left(1 - \frac{\alpha}{\beta}\right) \int_0^\infty  \int_0^\infty \int S_{s}(\bQ^{X(t)} \, | \, \PR_{X(t)}) \bQ(\dd X) \, \tau_\beta(\dd s) \, \tau_\alpha(\dd t).
	\end{align*}
	By our choice of $\bQ^{X(t)}$, we see that indeed 
	\begin{equation*}
	R(\beta) h(x) \geq R(\alpha) \left(R(\beta)h - \alpha \frac{R(\beta)h - h}{\beta} \right)(x),
	\end{equation*}
	which establishes \eqref{eqn:pseudo_resolvent_Markov_concatenate}.
\end{proof}

\subsubsection{Decomposing measures}\label{section:subsub_decomposition}

In this section, we will prove that 
\begin{equation} \label{eqn:pseudo_resolvent_Markov_decomposition}
R(\beta)h \leq R(\alpha) \left(R(\beta)h - \alpha \frac{R(\beta)h - h}{\beta} \right).
\end{equation}

The main step in the proof is to decompose the measure that turns up as the optimizer in the variational problem defining $R(\beta)h$. Fix $x \in E$ and let $\bQ \in \cP$ such that 
\begin{equation*}
R(\beta)h(x) = \int_0^\infty \int h(X(t)) \bQ(\dd X) \,  -S_t(\bQ \, | \, \PR_{x} )  \tau_\beta(\dd t).
\end{equation*}

By general measure theoretic arguments, we can find for every fixed $t$ a $\cF_t$ measurable family of measures $X \mapsto \bQ^{t,X}$ such that
\begin{equation} \label{eqn:measure_decomposition}
\bQ(\dd Y) = \int \bQ^{t,X}(\dd Y) \, \bQ(\dd X)
\end{equation}
and such that if $\bQ^{t,X}$ is restricted to trajectories up to time $t$ we find $\delta_{X}$. Denote by $\widehat{\bQ}^{t,X}$ the measure that is obtained from $\bQ^{t,X}$ under the push-forward map
\begin{equation*}
\theta^t(X)(s) = X(t+s).
\end{equation*}
Thus, $\widehat{\bQ}^{t,X}$ is supported by trajectories such that $Y(0) = X(t)$ (for $\bQ$ almost all $X$).

\begin{proof}[Proof of \eqref{eqn:pseudo_resolvent_Markov_decomposition}]
	
	As in Section \ref{section:subsub_concationation}, we obtain that
	\begin{multline*}
	\int_0^\infty \int h(X(t)) \bQ(\dd X) \tau_\beta(\dd t) \\
	= \frac{\alpha}{\beta} \int_0^\infty  \int h(X(t)) \bQ(\dd X) \tau_\alpha(\dd t) +  \left(1 - \frac{\alpha}{\beta}\right) \int_0^\infty  \int_0^\infty \int h(Y(s)) \widehat{\bQ}^{t,X}(\dd Y) \, \tau_\beta(\dd s) \, \tau_\alpha(\dd t). 
	\end{multline*}
	Thus, if we can prove that
	\begin{equation} \label{eqn:entropy_decomposition_for_optimizer}
	\begin{aligned}
	& \int_0^\infty S_t(\bQ \, | \, \PR_{x} )  \tau_\beta(\dd t) \\
	& = \int_0^\infty S_t(\bQ \, | \, \PR_{x} )  \tau_\alpha(\dd t) + \left(1 - \frac{\alpha}{\beta}\right) \int_0^\infty \int \int_0^\infty S_s(\widehat{\bQ}^{t,X} \, | \, \PR_{X(t)} ) \tau_\beta(\dd t) \, \bQ(\dd X) \, \tau_\alpha(\dd t)
	\end{aligned}
	\end{equation}
	then we obtain \eqref{eqn:pseudo_resolvent_Markov_decomposition} by replacing $\widehat{\bQ}^{t,X}$ by its optimum to obtain $R(\beta)h(X(t))$ in the integrand and afterwards optimizing to obtain $R(\alpha)$. This, however, follows as in the proof of the first inequality in Section \ref{section:subsub_concationation}.
%
%
\end{proof}

	\subsection{A variational semigroup generated by the resolvent} \label{subsection_resolvent_to_semigroup}
	
	We conclude this section by proving Proposition \ref{proposition:resolvents_approximate_semigroup}, that is, we establish that the resolvent approximates the semigroup.
	
	Again, the key idea is to reduce to a property of exponential distributions. This time, we will use that the sum of $n$ independent exponential random variables with mean $t/n$ converges to $t$. As the resolvent is defined in terms of an optimization procedure, we cannot directly apply this intuition. However, we will use natural upper and lower bounds for concatenations of $R(\lambda)$ that we can control.
	
	\begin{proposition}
		For each $h \in C_b(E)$, $t \geq 0$ and $x \in E$ we have
		\begin{equation*}
		\lim_{n \rightarrow \infty} R\left(\frac{t}{n}\right)^n h = V(t)h
		\end{equation*}
		for the strict topology.
	\end{proposition}
	
	The result will follow immediately from Lemma's \ref{lemma:T_inequalites} and \ref{lemma:T_weak_to_strict_continuity} below. We start with the definition of some additional operators. For each distribution $\tau \in \cP(\bR^+)$ and $h \in C_b(E)$, define
	\begin{align*}
	T^+(\tau)h(x) & := \int_0^\infty \sup_{\bQ \in \cP} \left\{\int h(X(t)) \bQ(\dd X)  - S_t(\bQ \, | \, \PR_x) \right\} \tau(\dd t), \\
	T^-(\tau) h(x) & := \sup_{\bQ \in \cP} \int_0^\infty \left\{ \int h(X(t)) \bQ(\dd X) - S_t(\bQ \, | \, \PR_x) \right\} \tau(\dd t).
	\end{align*}
	For all $\tau$ and $h$, we have $T^+(\tau) h \geq T^-(\tau) h$. For exponential random variables $\tau_\lambda$ or fixed times $t$, we find 
	\begin{align*}
	T^+(\tau_\lambda)h \geq R(\lambda)h = T^- h, \qquad T^+(\delta_t) h = V(t)h = T^-(\delta_t) h.
	\end{align*}

	\begin{lemma} \label{lemma:T_inequalites}
		For $\tau_1, \tau_2$, we have
		\begin{align*}
		T^+(\tau_1 \ast \tau_2)h & \geq T^+(\tau_1) T^+(\tau_2) h, \\
		T^-(\tau_1 \ast \tau_2)h & \leq T^-(\tau_1) T^-(\tau_2) h.
		\end{align*}
	\end{lemma}
	
	\begin{proof}
		The first claim follows by similar, but easier, arguments as in the proof of \eqref{eqn:pseudo_resolvent_Markov_concatenate} in Section \ref{section:subsub_concationation}. Similarly, for the second claim, we refer to the arguments in Section \ref{section:subsub_decomposition}.
	\end{proof}
	
	\begin{lemma} \label{lemma:T_weak_to_strict_continuity}
		Let $h \in C_b(E)$ and $t \in \bR^+$ and let $\tau_n \in \cP(\bR^+)$ be such that $\tau_n \rightarrow \delta_t$. 
		
		Then we have
		\begin{equation*}
		\lim_n T^+(\tau_n) h = T^+(\delta_t)h = V(t)h
		\end{equation*}
		for the strict topology. In addition, we have for each sequence $x_n \rightarrow x$ that
		\begin{equation*}
		\liminf_{n \rightarrow \infty} T^-(\tau_n) h(x_n) \geq T^-(\delta_t)h(x) = V(t)h(x)
		\end{equation*}
		as well as $\sup_n \vn{T^-(\tau_n) h} \leq \vn{h}$.
	\end{lemma}

	\begin{proof}
		Fix $h \in C_b(E)$ and a sequence $\tau_n$ and $t$ such that $\tau_n \rightarrow \delta_t$. Note that it is immediate that $\sup_n \vn{T^+(\tau_n)h} \leq \vn{h}$ and $\sup_n \vn{T^-(\tau_n)h} \leq \vn{h}$.
		
		\smallskip
		
		We proceed by establishing strict convergence for $T^+(\tau_n)h$. By Lemma \ref{lemma:DV_variational}, we have
		\begin{equation*}
		T^+(\tau_n)h(x) = \int_0^\infty V(t)h(x) \tau_n(\dd t).
		\end{equation*}
		By Lemma \ref{lemma:continuity_semigroup} the map $t \mapsto V(t)f$ is continuous for the strict topology. Thus strict continuity of $\tau \mapsto T^+(\tau)h$ follows.
		
		\smallskip

		For the second statement, fix $x_n$ converging to $x$ in $E$. Let $\bQ \in \cP(D_E(\bR^+))$ such that
		\begin{equation*}
		T^-(\delta_t)h(x) = V(t)h(x) = \int h(X(t)) \bQ(\dd X) - S_t(\bQ \, | \, \PR_x)
		\end{equation*}
		and such that $S_t(\bQ \, | \, \PR_x) = S(\bQ \, | \, \PR_x)$. 
		
		\smallskip
		
		By Proposition \ref{proposition:gamma_convergence_integral_entropy}, we can find $\bQ_n \in \cP_n$ such that $\bQ_n \rightarrow \bQ$ and such that for each $s$ we have $\limsup_s S_s(\bQ_n \, | \, \PR_{x_n}) \leq S_s(\bQ \, | \, \PR_x)$ and $S_s(\bQ_n \, | \PR_{x_n}) \leq S_s(\bQ \, | \, \PR_x) + 1$ for all $n$ and $s$. These properties imply that
		\begin{equation} \label{eqn:upper_bound_on_relative_entropy_in_proof_resolvent_to_semigroup}
		S_s(\bQ_n \, | \, \PR_{x_n}) \leq \begin{cases}
		S_t(\bQ_n \, | \, \PR_{x_n}) & \text{if } s < t+1, \\
		S_t(\bQ_n \, | \, \PR_{x_n}) + 1 & \text{if } s \geq t+1. \\
		\end{cases}
		\end{equation}	
		Thus, applying the $\liminf_n$ to $T^-(\tau_n)h(x_n)$, we find
		\begin{align*}
		\liminf_{n \rightarrow \infty}T^-(\tau_n) h(x_n) & \geq \liminf_n \int_0^\infty \left\{ \int h(X(s)) \bQ_n(\dd X) - S_s(\bQ_n \, | \, \PR_{x_n}) \right\} \tau_n(\dd s) \\
		& \geq  \liminf_{n \rightarrow \infty} \int_0^\infty  \int h(X(s)) \bQ_n(\dd X) \tau_n(\dd s) - \limsup_{n \rightarrow \infty} \int_0^\infty S_s(\bQ_n \, | \, \PR_{x_n}) \tau_n(\dd s). 
		\end{align*}
		As $\bQ_n \rightarrow \bQ$ and $\tau_n \rightarrow \tau$ and the map $s \mapsto X(s)$ is continuous at $t$ for $\bQ$ almost every $X$ as $\bQ \ll \PR_x$, the first term converges to $\int h(X(t)) \bQ(\dd X)$. For the second term, we obtain by \eqref{eqn:upper_bound_on_relative_entropy_in_proof_resolvent_to_semigroup} and the property that $S_t(\bQ \, | \PR_x) = S(\bQ \, | \, \PR_x)$
		\begin{align*}
		\limsup_{n \rightarrow \infty} \int_0^\infty S_s(\bQ_n \, | \, \PR_{x_n}) \tau_n(\dd s) & \leq \limsup_{n \rightarrow \infty}  S_{t+1}(\bQ_n \, | \, \PR_{x_n}) + \tau_n([t+1,\infty)) \\
		& \leq S_{t+1}(\bQ \, | \, \PR_x) \\
		& = S_t(\bQ \, | \, \PR_x).
		\end{align*}
		We conclude that $\liminf_n T^-(\tau_n)h(x_n) \geq V(t)f(x_n)$.
	\end{proof}

	\section{A large deviation principle for Markov processes} \label{section:Markov_LDP}

	In Section \ref{section:Markov_LDP_simple}, we considered a sequence of Markov processes on a Polish space $E$ and stated a large deviation principle on $D_E(\bR^+)$. In this section, we prove a more general version of this result that takes into account variations that one runs into in practice. As a first generalization, we consider Markov processes $t \mapsto X_n(t)$ on a sequence of spaces $E_n$ that are embedded into some space $E$ using maps $\eta_n : E_n \rightarrow E$.
	
	As an example $X_n$ could be a process on $E_n := \{-1,1\}^n$, whereas we are interested in the large deviation behaviour of the average of the $n$ values which takes values in $E = [-1,1]$.
	
	In Theorem \ref{theorem:LDP_simple}, we assumed exponential tightness and that certain sequences of functions converge. We need to modify these two concepts to allow for a sequence of spaces.
	\begin{itemize}
		\item We want to establish convergence of functions that are defined on different spaces. We therefore need a new notion of bounded and uniform convergence on compact sets. The key step in this definition will be to assign to each compact set $K \subseteq E$ a sequence of compact sets $K_n \subseteq E_n$ so that $\eta_n(K_n)$ `converge' to $K$. 
		In fact, to have a little bit more flexibility in our assignment of compact sets, we will work below with an large index set $\cQ$ so that to each $q \in \cQ$ we associate compact sets $K_n^q \subseteq E_n$ and $K^q \subseteq E$.
		\item Exponential tightness and buc convergence can be exploited together to make sure we get proper limiting statements. As our notion of buc convergence changes, we have to adapt our notion of exponential tightness to take into account the index set $\cQ$. 
	\end{itemize}
	We make to additional generalizations that are useful in practice.
	\begin{itemize}
		\item Often, it is hard to find an operator $H \subseteq C_b(E) \times C_b(E)$ that is a limit of $H_n \subseteq C_b(E_n) \times C_b(E_n)$. Rather one finds upper and lower bounds $H_\dagger$ and $H_\ddagger$ for the sequence $H_n$. See also Question \ref{question:upper_and_lower_bound} on whether at the pre-limit level one is able to work with upper and lower bounds.
		\item In the context of averaging or homogenisation, the natural limiting operator $H$ is a subset of $C_b(E) \times C_b(F)$, where $F$ is some space that takes into account additional information. For example $F = E \times \bR$, where the additional component $\bR$ takes into account the information of a fast process or a microscopic scale.
	\end{itemize}
	
	We thus start with a section on preliminaries that allows us to talk about these four extensions.

	\subsection{Preliminary definitions}

	Recall that $M(E)$ is the space of bounded measurable functions $f : E \rightarrow [-\infty,\infty]$. Denote
	\begin{align*}
	USC_u(E) & := \left\{f \in M(E) \, \middle| \,  f \, \text{upper semi-continuous}, \sup_x f(x) < \infty \right\}, \\
	LSC_l(E) & := \left\{f \in M(E) \, \middle| \, f \, \text{lower semi-continuous}, \inf_x f(x) > \infty \right\}.
	\end{align*}

		\begin{definition}[Kuratowski convergence]
		Let $\{A_n\}_{n \geq 1}$ be a sequence of subsets in a space $E$. We define the \textit{limit superior} and \textit{limit inferior} of the sequence as
		\begin{align*}
		\limsup_{n \rightarrow \infty} A_n & := \left\{x \in E \, \middle| \, \forall \, U \in \cU_x \, \forall \, N \geq 1 \, \exists \, n \geq N: \, A_n \cap U \neq \emptyset \right\}, \\
		\liminf_{n \rightarrow \infty} A_n & := \left\{x \in E \, \middle| \, \forall \, U \in \cU_x \, \exists \, N \geq 1 \, \forall \, n \geq N: \, A_n \cap U \neq \emptyset \right\}.
		\end{align*}
		where $\cU_x$ is the collection of open neighbourhoods of $x$ in $E$. If $A:= \limsup_n A_n = \liminf_n A_n$, we write $A = \lim_n A_n$ and say that $A$ is the Kuratowski limit of the sequence $\{A_n\}_{n \geq 1}$.
	\end{definition}

%
%
%
%
%


	\subsubsection{Embedding spaces}
	
	Our main result will be based on the following setting.
	
	\begin{assumption} \label{assumption:embedding_spaces}
		We have spaces $E_n$ and $E, F$ and continuous maps $\eta_n : E_n \rightarrow E$, $\widehat{\eta}_n : E_n \rightarrow F$ and a continuous surjective map $\gamma : F \rightarrow E$ such that the following diagram commutes: \\
		\begin{center}
			\begin{tikzpicture}
			\matrix (m) [matrix of math nodes,row sep=1em,column sep=4em,minimum width=2em]
			{
				{ }   & F \\
				E_n & { } \\
				{ }   & E \\};
			\path[-stealth]
			(m-2-1) edge node [above] {$\widehat{\eta}_n$} (m-1-2)
			(m-2-1) edge node [below] {$\eta_n$} (m-3-2)
			(m-1-2) edge node [right] {$\gamma$} (m-3-2);
			\end{tikzpicture}
		\end{center}
		
		In addition, there is a directed set $\cQ$ (partially ordered set such that every two elements have an upper bound). For each $q \in \cQ$, we have compact sets $K_n^q \subseteq E_n$ and compact sets $K^q \subseteq E$ and $\widehat{K}^q \subseteq F$ such that
		\begin{enumerate}
			\item If $q_1 \leq q_2$, we have $K^{q_1} \subseteq K^{q_2}$, $\widehat{K}^{q_1} \subseteq \widehat{K}^{q_2}$ and for all $n$ we have $K_{n}^{q_1} \subseteq K_n^{q_2}$.
			\item \label{item:assumption_abstract_2_limit_compact} For all $q \in \cQ$ we have $\bigcup_n \widehat{\eta}_n(K_n^q) \subseteq \widehat{K}^q$.
			\item \label{item:assumption_abstract_2_exists_q} For each compact set $K \subseteq E$, there is a $q \in \cQ$ such that
			\begin{equation*}
			K \subseteq \liminf_n \eta_n(K_n^q).
			\end{equation*}
			\item \label{item:assumption_abstract_2_gamma_mapsintoeachother} We have $\gamma(\widehat{K}^q) \subseteq K^q$.
		\end{enumerate}
	\end{assumption}

\begin{remark}
	Note that \ref{item:assumption_abstract_2_limit_compact} implies that $\limsup_n \widehat{\eta}_n(K_n^q) \subseteq \widehat{K}^q$ and together with \ref{item:assumption_abstract_2_gamma_mapsintoeachother} that $\limsup_n \eta_n(K_n^q) \subseteq K^q$.
	
	Thus, the final three conditions imply that the sequences $\eta_n(K^q_n)$ for various $q \in \cQ$ covers all compact sets in $E$, and also are covered by compact sets in $E$ (in fact this final statement holds on the larger space $F$). This implies that the index set $\cQ$ connects the structure of compact sets in $E$ and $F$ in a suitable way to (a subset) of the compact sets of the sequence $E_n$.
\end{remark}

%

	We use our index set $\cQ$ to extend our notion of bounded and uniform convergence on compacts sets.

	\begin{definition} \label{definition:abstract_LIM}
		Let Assumption \ref{assumption:embedding_spaces} be satisfied. For each $n$ let $f_n \in M_b(E_n)$ and $f \in M_b(E)$. We say that $\LIM f_n = f$ if
		\begin{itemize}
			\item $\sup_n  \vn{f_n} < \infty$,
			\item if for all $q \in \cQ$ and $x_n \in K_n^q$ converging to $x \in K^q$ we have
			\begin{equation*}
			\lim_{n \rightarrow \infty} \left|f_n(x_n) - f(x)\right| = 0.
			\end{equation*}
		\end{itemize}
	\end{definition}

	\begin{remark} \label{remark:LIM_equivalence}
		Note that if $f \in C_b(E)$ and $f_n \in M_b(E_n)$, we have that $\LIM f_n = f$ if and only if
		\begin{itemize}
			\item $\sup_n  \vn{f_n} < \infty$,
			\item if for all $q \in \cQ$ 
			\begin{equation*}
			\lim_{n \rightarrow \infty} \sup_{x \in K_n^q} \left|f_n(x) - f(\eta_n(x))\right| = 0.
			\end{equation*}
		\end{itemize}
	\end{remark}

	\subsubsection{Viscosity solutions of Hamilton-Jacobi equations}

	Below we will introduce a more general version of viscosity solutions compared to Section \ref{section:preliminaries}. One recovers the old definition by taking $B_\dagger = B_\ddagger = B$, $F = E$ and $\gamma(x) = x$.
	
	\begin{definition}
		Let $B_\dagger \subseteq LSC_l(E) \times USC_u(F)$ and $B_\ddagger \subseteq USC_u(E) \times LSC_l(F)$. Fix $h_1,h_2 \in C_b(E)$. Consider the equations
		\begin{align} 
		f -  B_\dagger f & = h_1, \label{eqn:differential_equation_Bdagger} \\
		f - B_\ddagger f & = h_2. \label{eqn:differential_equation_Bddagger}
		\end{align}
		\begin{itemize}
			\item We say that $u : X \rightarrow \bR$ is a \textit{subsolution} of equation \eqref{eqn:differential_equation_Bdagger} if $u \in USC_u(E)$ and if, for all $(f,g) \in B_\dagger$ such that $\sup_x u(x) - f(x) < \infty$ there is a sequence $y_n \in F$ such that
			\begin{equation} \label{eqn:subsol_optimizing_sequence}
			\lim_{n \rightarrow \infty} u(\gamma(y_n)) - f(\gamma(y_n))  = \sup_x u(x) - f(x),
			\end{equation}
			and
			\begin{equation} \label{eqn:subsol_sequence_outcome}
			\limsup_{n \rightarrow \infty} u(\gamma(y_n)) - g(y_n) - h_1(\gamma(y_n)) \leq 0.
			\end{equation}
			\item 	We say that $v : E \rightarrow \bR$ is a \textit{supersolution} of equation \eqref{eqn:differential_equation_Bddagger} if $v \in LSC_l(E)$ and if, for all $(f,g) \in B_\ddagger$ such that $\inf_x v(x) - f(x) > - \infty$ there is a sequence $y_n \in Y$ such that
			\begin{equation} \label{eqn:supersol_optimizing_sequence}
			\lim_{n \rightarrow \infty} v(\gamma(y_n)) - f(\gamma(y_n))  = \inf_x v(x) - f(x),
			\end{equation}
			and
			\begin{equation} \label{eqn:supersol_sequence_outcome}
			\liminf_{n \rightarrow \infty} v(\gamma(y_n)) - g(y_n) - h_2(\gamma(y_n)) \geq 0.
			\end{equation}
			\item We say that $u$ is a \textit{solution} of the pair of equations \eqref{eqn:differential_equation_Bdagger} and \eqref{eqn:differential_equation_Bddagger} if it is both a subsolution for $B_\dagger$ and a supersolution for $B_\ddagger$.
			\item We say that  \eqref{eqn:differential_equation_Bdagger} and \eqref{eqn:differential_equation_Bddagger} satisfy the \textit{comparison principle} if for every subsolution $u$ to \eqref{eqn:differential_equation_Bdagger} and supersolution $v$ to \eqref{eqn:differential_equation_Bddagger}, we have
			\begin{equation} \label{eqn:comparison_estimate}
			\sup_x u(x) - v(x) \leq \sup_x h_1(x) - h_2(x).
			\end{equation}
			If $B = B_\dagger = B_\ddagger$ and $h = h_1 = h_2$, we will say that the comparison principle holds for $f - \lambda Bf = h$, if for any subsolution $u$ for $f - \lambda Bf = h_1$ and supersolution $v$ of $f - \lambda Bf = h_2$ the estimate in \eqref{eqn:comparison_estimate} holds. 
		\end{itemize}
	\end{definition}

	\subsubsection{Notions of convergence of Hamiltonians}

	We now introduce our notion of upper and lower bound for the sequence $H_n$.
	
	\begin{definition} \label{definition:extended_sub_super_limit}
		Consider the setting of Assumption \ref{assumption:embedding_spaces}.  Suppose that for each $n$ we have operators $H_{n} \subseteq C_b(E_n) \times C_b(E_n)$.
		\begin{enumerate}
			\item The \textit{extended sub-limit} $ex-\subLIM_n H_{n}$ is defined by the collection $(f,g) \in H_\dagger \subseteq LSC_l(E) \times USC_u(F)$ such that there exist $(f_n,g_n) \in H_{n}$ satisfying
			\begin{gather} 
			\LIM f_n \wedge c = f \wedge c, \qquad \forall \, c \in \bR, \label{eqn:convergence_condition_sublim_constants} \\
			\sup_{n} \sup_{x \in X_n} g_n(x) < \infty, \label{eqn:convergence_condition_sublim_uniform_gn}
			\end{gather}
			and if for any $q \in \cQ$ and sequence $z_{n(k)} \in K_{n(k)}^q$ (with $k \mapsto n(k)$ strictly increasing) such that $\lim_{k} \widehat{\eta}_{n(k)}(z_{n(k)}) = y$ in $F$ with $\lim_k f_{n(k)}(z_{n(k)}) = f(\gamma(y)) < \infty$ we have
			\begin{equation} \label{eqn:sublim_generators_upperbound}
			\limsup_{k \rightarrow \infty}g_{n(k)}(z_{n(k)}) \leq g(y).
			\end{equation}
			\item The \textit{extended super-limit} $ex-\superLIM_n H_{n}$ is defined by the collection $(f,g) \in H_\ddagger \subseteq USC_u(E) \times LSC_l(F)$ such that there exist $(f_n,g_n) \in H_{n}$ satisfying
			\begin{gather} 
			\LIM f_n \vee c = f \vee c, \qquad \forall \, c \in \bR, \label{eqn:convergence_condition_superlim_constants} \\
			\inf_{n} \inf_{x \in X_n} g_n(x) > - \infty, \label{eqn:convergence_condition_superlim_uniform_gn}
			\end{gather}
			and if for any $q \in \cQ$ and sequence $z_{n(k)} \in K_{n(k)}^q$ (with $k \mapsto n(k)$ strictly increasing) such that $\lim_{k} \widehat{\eta}_{n(k)}(z_{n(k)}) =y$ in $F$ with $\lim_k f_{n(k)}(z_{n(k)}) = f(\gamma(y)) > - \infty$ we have
			\begin{equation}\label{eqn:superlim_generators_lowerbound}
			\liminf_{k \rightarrow \infty}g_{n(k)}(z_{n(k)}) \geq g(y).
			\end{equation}
		\end{enumerate}
		
	\end{definition}

	\begin{remark}
		The conditions in \eqref{eqn:convergence_condition_sublim_constants} and \eqref{eqn:convergence_condition_superlim_constants} are implied by $\LIM f_n = f$. Conditions
		\eqref{eqn:convergence_condition_sublim_uniform_gn} and \eqref{eqn:sublim_generators_upperbound} are implied by $\LIM_n g_n \leq g$ whereas conditions \eqref{eqn:convergence_condition_superlim_uniform_gn} and \eqref{eqn:superlim_generators_lowerbound} are implied by $\LIM_n g_n \geq g$.
		
		Comparing this to Definition \ref{definition_extended_limit}, we indeed see that the sub and super-limit can be interpreted as upper and lower bounds instead of limits.
	\end{remark}

	\subsection{Large deviations for Markov process}

	We proceed by stating our main large deviation result, which extends Theorem \ref{theorem:LDP_simple}. We first give the appropriate generalization of Condition \ref{condition:simple_markov_ldp}.

	\begin{condition} \label{condition:markov_ldp}
		Let $E_n, E, F, \eta_n,\hat{\eta}_n,\gamma$ be as in Assumption \ref{assumption:embedding_spaces}.
		
		Let $A_n \subseteq C_b(E_n) \times C_b(E_n)$ be linear operators and let $r_n$ be positive real numbers such that $r_n \rightarrow \infty$. Suppose that
		\begin{itemize}
			\item The martingale problems for $A_n$ are well-posed on $E_n$. Denote by $x \mapsto \PR_x^n$ the solution to the martingale problem for $A_n$.
			\item For each $n$ that $x \mapsto \PR_x^n$ is continuous for the weak topology on $\cP(D_{E_n}(\bR^+))$.
			\item For each $a_1 > 0$ there is a $q \in \cQ$ such that
			\begin{equation*}
			\limsup_{n \rightarrow \infty} \frac{1}{r_n} \log \PR\left[ Y_n(0) \notin K_n^{q} \, \middle| \, Y_n(0) = y \right] \leq - a_1.
			\end{equation*}
			\item \, [Exponential compact containment] For each $q \in \cQ$, $T > 0$ and $a_2 > 0$ there exists $\hat{q} = \hat{q}(q,T,a_2) \in \cQ$ such that 
			\begin{equation*}
			\limsup_{n \rightarrow \infty} \sup_{y \in K_n^q} \frac{1}{r_n} \log \PR\left[\exists \, t \leq T: \, Y_n(t) \notin K_n^{\hat{q}} \, \middle| \, Y_n(0) = y \right] \leq - a_2.
			\end{equation*}
		\end{itemize}
	\end{condition}
	
	Note that these conditions can be mapped to the ones of Condition \ref{condition:simple_markov_ldp}, except for the third one. In Theorem \ref{theorem:LDP_simple}, we assumed the large deviation principle at time $0$ which implies this remaining condition if $E_n = E$. Here, however, we need to assume that the mass is concentrated already on a $q \in \cQ$ before the maps $\eta_n$.

	\begin{theorem} \label{theorem:LDP}
		Suppose that we are in the setting of Assumption \ref{assumption:embedding_spaces} and that Condition \ref{condition:markov_ldp} is satisfied. Denote $X_n = \eta_n(Y_n)$. Define the operator semigroup $V_n(t)$ on $C_b(E_n)$:
		\begin{equation*}
		V_n(t)f(y) := \frac{1}{r_n} \log \bE\left[e^{r_n f(Y_n(t))} \, \middle| \, Y_n(0) = y\right].
		\end{equation*}
		Let $H_n$ be the set of pairs $(f,g) \in C_b(E_n) \times C_b(E_n)$ such that
		\begin{equation*}
		t \mapsto \exp\left\{r_n \left( f(Y_n(t)) - f(Y_n(0)) - \int_0^t g(Y_n(s)) \dd s \right)  \right\}
		\end{equation*}
		are martingales with respect to $\cF_t^n := \sigma\{ Y_n(s) \, | \, s \leq t\}$. Suppose furthermore that
		\begin{enumerate}
			\item The large deviation principle holds for $X_n(0) = \eta_n(Y_n(0))$ with speed $r_n$ and good rate function $I_0$.
			\item The processes $X_n = \eta_n(Y_n)$ are exponentially tight on $D_E(\bR^+)$. 
			\item There are two operators $H_\dagger \subseteq LSC_l(E) \times USC_u(F)$ and $H_\ddagger \subseteq USC_l(E) \times LSC_u(F)$ such that $H_\dagger \subseteq ex-\LIMSUP H_n$ and $H_\ddagger \subseteq ex-\LIMINF H_n$.
			\item Let $D \subseteq C_b(E)$ be a quasi-dense subset of $C_b(E)$. Suppose that for all $h \in D$ and $\lambda > 0$ the comparison principle holds for viscosity subsolutions to $f - \lambda H_\dagger f = h$ and supersolutions to $f - \lambda H_\ddagger f = h$.
		\end{enumerate}
	Then there is a semigroup $V(t)$ on $C_b(E)$ such that if $\LIM f_n = f$ and $t_n \rightarrow t$ we have $\LIM V_n(t_n)f_n = V(t)f$.	In addition, the processes $X_n = \eta_n(Y_n)$ satisfy a large deviation principle on $D_E(\bR^+)$ with speed $r_n$ and rate function
		\begin{equation} \label{eqn:LDP_rate2}
		I(\gamma) = I_0(\gamma(0)) + \sup_{k \geq 1} \sup_{\substack{0 = t_0 < t_1 < \dots, t_k \\ t_i \in \Delta_\gamma^c}} \sum_{i=1}^{k} I_{t_i - t_{i-1}}(\gamma(t_i) \, | \, \gamma(t_{i-1})).
		\end{equation}
		Here $\Delta_\gamma^c$ is the set of continuity points of $\gamma$. The conditional rate functions $I_t$ are given by
		\begin{equation*}
		I_t(y \, | \, x) = \sup_{f \in C_b(E)} \left\{f(y) - V(t)f(x) \right\}.
		\end{equation*}
		
	\end{theorem}

	We proceed with a two remarks on how to obtain exponential tightness of the processes and the variational representation of the rate function.
	
	\smallskip
	
	We start with the exponential tightness. The verification of exponential tightness of the processes $\eta_n(Y_n)$ comes down to verifying two statements. The first one is exponential compact containment, which has been assumed in Condition \ref{condition:markov_ldp}. The second one is to control the oscillations of the process, which can often be achieved by considering the exponential martingales. This has been done in the proof of Corollary 4.19 of \cite{FK06}. We state it for completeness, including a definition that we need in its statement.

\begin{definition}
	Let $q$ be a metric that generates the topology on $E$. We say that $\fD \subseteq C_b(E)$ \textit{approximates} the metric $q$ if for each compact $K \subseteq E$ and $z \in K$ there exist $f_n \in \fD$ such that
	\begin{equation*}
	\lim_n \sup_{x \in K} \left| f_n(x) - q(x,q)\right| = 0.
	\end{equation*}
\end{definition}

\begin{proposition}[Corollary 4.19 \cite{FK06}]
	Suppose that we are in the setting of Assumption \ref{assumption:embedding_spaces} and Condition \ref{condition:markov_ldp}. Let $r_n > 0$ be some sequence such that $r_n \rightarrow \infty$. Denote $X_n = \eta_n(Y_n)$. Let $\fD \subseteq C_b(E)$ and $\fS \subseteq \bR$. Suppose that
	\begin{enumerate}
		\item Either $\fF$ is closed under addition and separates points in $E$ and $\fS=\bR$ or $\fF$ approximates a metric $q$ and $\fS = (0,\infty)$.
		\item \label{item:convergence_of_functions_compact_containment} For each $\lambda \in \fS$ and $f \in \fD$ there are $(f_n,g_n)$ such that $(\lambda f_n,g_n) \in \cD(H_n)$ with $\LIM f_n = f$ and for all $q \in \cQ$
		\begin{equation*}
		\sup_n \sup_{x \in K_n^q} g_n(x) < \infty.
		\end{equation*}
	\end{enumerate} 
Then the sequence of processes $\{X_n\}$ is exponentially tight.
\end{proposition}

	Note that Condition \ref{item:convergence_of_functions_compact_containment} often follows from the convergence $H_\dagger \subseteq ex-\LIMSUP_n H_n$.
	
	We proceed with a remark on the variational representation of the rate function.
	
\begin{remark}
	For an expression of the large deviation rate-functional in a Lagrangian form, one can show that a variational resolvent, similar to the one in this paper, but with a Lagrangian instead of an entropy as a penalization, solves the limiting Hamilton-Jacobi equation. This has been carried out in Chapter 8 of \cite{FK06}. Generally, this leads to an expression
	\begin{equation*}
	I(\gamma) = \begin{cases}
	I_0(\gamma(0)) + \int \cL(\gamma(s),u) \nu(\dd u, \dd s) & \text{if } \gamma \text{ is absolutely continuous}  \\
	\infty & \text{otherwise}.
	\end{cases}
	\end{equation*}
	$\cL$ can usually be obtained from the operators $H_\dagger$ and $H_\ddagger$ by a (Legendre) transformation. Often one formally has $H_\dagger f(x) = \cH(x,\dd f(x))$ and $H_\ddagger f(x) = \cH(x, \dd f(x))$. $\cL$ is then obtained as $\cL(x,v) = \sup_p \ip{p}{v} - \cH(x,p)$.
	
	We refrain from carrying out this step as it would follow \cite[Chapter 8]{FK06} exactly.
\end{remark}

	\subsection{Strategy of the proof and discussion on the method of proof}

	Feng and Kurtz \cite{FK06} showed in their extensive monograph that path-space large deviations of the processes $X_n = \eta_n(Y_n)$ on $D_E(\bR^+)$ can be obtained by establishing exponential tightness and the convergence of the non-linear semigroups $V_n(t)$.
	
	We repeat the important steps in this approach.
	
	\begin{enumerate}[(1)]
		\item A projective limit theorem (rather a special version of the projective limit theorem and the inverse contraction principle, \cite[Theorem 4.28]{FK06}) for the Skorokhod space establishes that, given exponential tightness, it suffices to establish large deviations for the finite dimensional distributions of $X_n = \eta_n(Y_n)$.
		\item By Bryc's theorem, the large deviations for finite dimensional distributions follow from the convergence of the rescaled log-Laplace transforms. 
		\item Using the Markov property, one can reduce the convergence of the log-Laplace transforms to the large deviation principle at time $0$ and the convergence of semigroups. 
	\end{enumerate}
	
	We will give a new proof of the path-space large deviation principle on the basis of this strategy. However, the key component of establishing the convergence of semigroup will be based on the explicit identification of the resolvents of the non-linear semigroups and the semigroup convergence result of \cite{Kr19}.
	
	\smallskip
	
	At this point we remark two differences with the main result of \cite{FK06}. 
	
	Throughout we assume that the maps $\eta_n,\hat{\eta}_n$ and $x \mapsto \PR_x^n$ are continuous, whereas in \cite{FK06} they are allowed to be measurable only. The results in \cite{Kr19} allow one to work with measurable resolvents also, but the methods of the first part of this paper are based on properties of continuous functions. It would be of interest to see whether these methods can be extended to the context of measurable functions also.
	
	The key point why \cite{FK06} can work with measurable maps is the approximation of the processes $X_n$ by their Yosida approximants. This approximation does introduce an extra condition into the notions of $ex-\LIMSUP$ and $ex-\LIMINF$. Compare our \ref{eqn:convergence_condition_sublim_uniform_gn} and \ref{eqn:convergence_condition_superlim_uniform_gn} to Equations (7.19) and (7.22) of \cite{FK06}.

	\subsection{Proof of Theorem \ref{theorem:LDP}}

	The following result is based on the variant of the projective limit theorem and Bryc's theorem. See Theorem 5.15, Remark 5.16 and Corollary 5.17 in \cite{FK06}.

	\begin{theorem}\label{theorem:LDP_basic}
		Suppose that we are in the setting of Assumption \ref{assumption:embedding_spaces} and that Condition \ref{condition:markov_ldp} is satisfied. Denote $X_n = \eta_n(Y_n)$. Define the operator semigroup $V_n(t)$ on $C_b(E_n)$:
		\begin{equation*}
		V_n(t)f(y) := \frac{1}{r_n} \log \bE\left[e^{r_n f(Y_n(t))} \, \middle| \, Y_n(0) = y\right].
		\end{equation*}
		Suppose furthermore that
		\begin{enumerate}
			\item The large deviation principle holds for $X_n(0) = \eta_n(Y_n(0))$ with speed $r_n$ and good rate function $I_0$.
			\item The processes $X_n = \eta_n(Y_n)$ are exponentially tight on $D_E(\bR^+)$. 
			\item There is a semigroup $V(t)$ on $C_b(E)$ such that if $\LIM f_n = f$ and $t_n \rightarrow t$ we have $\LIM V_n(t_n)f_n = V(t)f$.
		\end{enumerate}
		Then the processes $X_n = \eta_n(Y_n)$ satisfy a large deviation principle on $D_E(\bR^+)$ with speed $r_n$ and rate function
		\begin{equation} \label{eqn:LDP_rate2_basic}
		I(\gamma) = I_0(\gamma(0)) + \sup_{k \geq 1} \sup_{\substack{0 = t_0 < t_1 < \dots, t_k \\ t_i \in \Delta_\gamma^c}} \sum_{i=1}^{k} I_{t_i - t_{i-1}}(\gamma(t_i) \, | \, \gamma(t_{i-1})).
		\end{equation}
		Here $\Delta_\gamma^c$ is the set of continuity points of $\gamma$. The conditional rate functions $I_t$ are given by
		\begin{equation*}
		I_t(y \, | \, x) = \sup_{f \in C_b(E)} \left\{f(y) - V(t)f(x) \right\}.
		\end{equation*}
	\end{theorem}

	We will not prove this result, but refer to \cite[pages 93 and 94]{FK06} as it follows from essentially the projective limit theorem and Brycs result. The new contribution of this paper is a new method to obtain the convergence of semigroups based on the explicit identification of the resolvent corresponding to the semigroups $V_n(t)$.

	\begin{proof}[Proof of Theorem \ref{theorem:LDP}]
		The result follows from Theorem \ref{theorem:LDP_basic} if we can establish the convergence of semigroups, and obtain a limiting semigroup that is defined on all of $C_b(E)$. To do so, we apply Theorem 6.1 in \cite{Kr19}. The semigroups $V_n$ are of the type as in Remark \ref{remark:rescaling_of_operators}, whose resolvents and generators we have identified in Theorem \ref{theorem:Markov_solving_HJ} and Proposition \ref{proposition:resolvents_approximate_semigroup}.
		
		The conditions on convergence of Hamiltonians for \cite[Theorem 6.1]{Kr19} have been assumed in Theorem \ref{theorem:LDP} and we can work with $B_n = C_b(E_n)$ due to Proposition \ref{proposition:continuity_of_resolvent}. 
		
		The following two ingredients for the application of \cite[Theorem 6.1]{Kr19} are missing
		\begin{itemize}
			\item joint local equi-continuity of the semigroups $\{V_n(t)\}_{n\geq 1}$,
			\item joint local equi-continuity of the resolvents $\{R_n(\lambda)\}_{n \geq 1}$,
		\end{itemize}
		We check these properties in Lemmas \ref{lemma:equi_cont_of_semigroup} and \ref{lemma:equi_cont_of_resolvent} below.
		
		As a consequence \cite[Theorem 6.1]{Kr19} can be applied, and we obtain convergence of $V_n(t)$ to a semigroup $V(t)$, which is defined on the quasi-closure of the set
		\begin{equation*}
		\bigcup_{\lambda > 0} \left\{R(\lambda) h \, \middle| \, h \in C_b(E) \right\}.
		\end{equation*}
		Thus, if for all $h \in C_b(E)$ we have $\lim_{\lambda \rightarrow 0} R(\lambda) h = h$ for the strict topology, then indeed the semigroup $V(t)$ is defined on all of $C_b(E)$. We prove this in Lemma \ref{lemma:continuity_limiting_resolvent_at_0} below.
		
		This establishes the final result.
	\end{proof}

	The estimates below will be similar in spirit to estimates carried out in Section \ref{section:regularity}. There we were able to use tightness of sets of measures that have bounded relative entropy (see Proposition \ref{proposition:equi_coercivity_relative_entropy}). Here, however, we need an argument that allows us to obtain tightness in the sense of estimates with the index set $\cQ$ from exponential compact containment condition and rescaled boundedness of relative entropies. A basic estimate of this type is included as Proposition \ref{proposition:exponential_tightness_and_entropy_control_implies_tightness} and will serve as the key replacement of Proposition \ref{proposition:equi_coercivity_relative_entropy}.

	\begin{lemma} \label{lemma:equi_cont_of_semigroup}
		Consider the setting of Theorem \ref{theorem:LDP}.	The semigroups $V_n(t)$  are  locally strictly equi-continuous on bounded sets: for all $q \in \cQ$, $\delta > 0$ and $T > 0$, there is a $\hat{q} \in \cQ$ such that for all $n \geq 1$, $h_{1},h_{2} \in C_b(E_n)$ and $0 \leq t \leq T$ we have
		\begin{equation*}
		\sup_{y \in K_n^q} \left\{ V_n(t)h_{1}(y) - V_n(t)h_{2}(y) \right\} \leq \delta \sup_{x \in E_n} \left\{ h_{1}(x) - h_{2}(x) \right\} + \sup_{y \in K^{\hat{q}}_n} \left\{ h_{1}(y) - h_{2}(y) \right\}.
		\end{equation*}
	\end{lemma}

	\begin{proof}
		Fix $h_1,h_2 \in C_b(E_n)$, $\delta > 0$, $q \in \cQ$, and $T > 0$. By exponential compact containment, see Condition \ref{condition:markov_ldp}, there is $\hat{q}$ such that 
		\begin{equation} \label{eqn:compact_containment_in_equicont_of_V}
		\limsup_{n \rightarrow \infty} \sup_{y \in K_n^q} \frac{1}{r_n} \log \PR\left[\exists \, t \leq T: \, Y_n(t) \notin K_n^{\hat{q}} \, \middle| \, Y_n(0) = y \right] \leq - a.
		\end{equation}
		Fix $n$, $y \in K_n^q$ and $t \leq T$. Consider the definition of $V(t)h_2$:
		\begin{equation*}
		V_n(t)h_2(y) = \sup_{\bQ \in \cP(D_E(\bR^+))} \left\{ \int h_2(Y_n(t)) \bQ(\dd Y_n) - \frac{1}{r_n} S(\bQ \, | \, \PR_n) \right\}.
		\end{equation*}
		As $h_2$ is bounded, the optimizer $\bQ_n$ must satisfy $\frac{1}{r_n} S(\bQ_n \, | \, \bP_n) \leq 2 \vn{h_2}$. Thus by Proposition \ref{proposition:exponential_tightness_and_entropy_control_implies_tightness} and \eqref{eqn:compact_containment_in_equicont_of_V} applied to $\bQ_n$ restricted to the marginal at time $t$, we have that for each $\delta = 2\varepsilon > 0$, there is a $\bar{q}$ such that 
		\begin{equation*}
		\sup_n \bQ_n[Y_n(t) \notin K_n^{\bar{q}})] \leq \delta.
		\end{equation*}
		It follows that
		\begin{equation*}
		\sup_{y \in K_n^q} V(t)h_1(y) - V(t)h_2(y) \leq \delta \sup_{x \in E_n} h_1(x) - h_2(x) + \sup_{y \in K_n^{\bar{q}}} h_1(y) - h_2(y).
		\end{equation*}

	\end{proof}

	\begin{lemma} \label{lemma:equi_cont_of_resolvent}
		Consider the setting of Theorem \ref{theorem:LDP}. The resolvents $R_n(\lambda)$ are locally strictly equi-continuity on bounded sets: for all $q \in \cQ$, $\delta > 0$ and $\lambda_0 > 0$, there is a $\hat{q} \in \cQ$ such that for all $n$ and $h_{1},h_{2} \in C_b(E_n)$ and $0 < \lambda \leq \lambda_0$ that
		\begin{equation*}
		\sup_{y \in K_n^q} \left\{ R_n(\lambda)h_{1}(y) - R_n(\lambda)h_{2}(y) \right\} \leq \delta \sup_{x \in E_n} \left\{ h_{1}(x) - h_{2}(x) \right\} + \sup_{y \in K^{\hat{q}}_n} \left\{ h_{1}(y) - h_{2}(y) \right\}.
		\end{equation*}
	\end{lemma}
	
	\begin{proof}
		If we work for a single $\lambda$ instead of uniformly over $0 < \lambda \leq \lambda_0$, we can proceed as in the proof above. We first cut-off the tail of the exponential random variable which introduces a small error. Then we use the exponential compact containment condition and Proposition \ref{proposition:exponential_tightness_and_entropy_control_implies_tightness} to find an appropriate $\hat{q}$ that can be used to finish the argument as in the proof of Lemma \ref{lemma:equi_cont_of_semigroup} above.
		
		If we work with a uniform estimate over $0 < \lambda \leq \lambda_0$, the argument needs to be adapted as in Lemma \ref{lemma:strict_equicontinuity_of_resolvent}. We carry out a similar adaptation in the proof of Lemma \ref{lemma:continuity_limiting_resolvent_at_0} below.
	\end{proof}

\begin{lemma} \label{lemma:continuity_limiting_resolvent_at_0}
	Consider the setting of Theorem \ref{theorem:LDP}. For all $h \in C_b(E)$, we have $\lim_{\lambda \rightarrow 0} R(\lambda)h = h$ for the strict topology.
\end{lemma}

\begin{remark}
	Below we give a proof based on an approximation argument and a variant of Lemma \ref{lemma:continuity_resolvent_at_0}. A less direct proof can be given also by using that the domains $\cD(H_\dagger)$ and $\cD(H_\ddagger)$ are sufficiently rich. See Proposition 7.1 of \cite{Kr19}
\end{remark}

\begin{proof}
	Fix $h \in C_b(E)$.	First of all, $\sup_\lambda \vn{R(\lambda)h} \leq \vn{h}$, so it suffices to establish uniform convergence on compact sets. Thus, fix a compact set $K \subseteq E$. We prove
	\begin{equation*}
	\lim_{\lambda \downarrow 0} \sup_{x \in K} \left|R(\lambda)h(x) - h(x)\right|= 0.
	\end{equation*}
	Fix $q \in \cQ$ such that $K \subseteq K^q$ and set $h_n = h \circ \eta_n$. Then we have by construction that $\LIM h_n = h$ and by Theorem 6.1 of \cite{Kr19} we have $\LIM R_n(\lambda)h_n = R(\lambda)h$ for any $\lambda > 0$.
	
	Pick $x \in K$ and let $x_n \in K_n^q$ such that $\eta_n(x_n) \rightarrow x$. We have
	\begin{align*}
	R(\lambda)h(x) - h(x) & = R(\lambda)h(x) -  R_n(\lambda)h_n(x_n) \\
	& \qquad + R_n(\lambda)h_n(x_n) - h_n(x_n) \\
	& \qquad + h_n(x_n) - h(x).
	\end{align*}
	Thus the result follows if we can prove that for each $\varepsilon > 0$ there is a $\lambda$ such that
	\begin{equation} \label{eqn:uniform_bound_resolvents}
	\sup_n \left| R_n(\lambda)h_n(x_n) - h_n(x_n)\right| \leq \varepsilon.
	\end{equation}

	Denote by $\PR_y^n$ the law of $Y_n$ on $D_{E_n}(\bR^+)$ when started in $y \in E_n$. We have
	\begin{multline} \label{eqn:expression_Resolvent_to_0}
	R_n(\lambda)h_n(x_n) - h_n(x_n) \\
	= \sup_{\bQ \in \cP(D_{E_n}(\bR^+)} \int_0^\infty h(\eta_n(Y_n(t))) - h(\eta_n(x_n)) \bQ(\dd Y_n) - \frac{1}{r_n} S_t(\bQ \, | \, \PR_{x_n}^n) \tau_{\lambda}(\dd t).
	\end{multline}
	As in Lemma \ref{lemma:continuity_resolvent_at_0}, we argue via a lower and upper bound.
	\begin{align*}
	R_n(\lambda)h_n(x_n) - h_n(x_n) & \geq \int_0^\infty h(\eta_n(Y_n(t)) - h(\eta_n(x_n)) \PR_{x_n}^n(\dd Y_n) \tau_{\lambda}(\dd t) \\
	& = \int_0^\infty h(X_n(t))) - h(X_n(0)) \PR_{x_n}^n \circ \eta_n^{-1}(\dd X_n) \tau_{\lambda}(\dd t)
	\end{align*}
	As the processes $X_n$ are exponentially tight by assumption, the upper bound follows as tightness gives control on the modulus of continuity of the trajectories $X_n$.
	
	\smallskip
	
	For the lower bound, we can find $\bQ_{n,\lambda} \in D_{E_n}(\bR^+)$ such that 
	\begin{equation} \label{eqn:control_on_entropies}
	\frac{1}{r_n}\int_0^\infty S_t(\bQ_{n,\lambda} \, | \, \PR_{x_n}^n) \tau_\lambda(\dd t) \leq 2\vn{h}
	\end{equation}
	that achieve the suprema in \eqref{eqn:expression_Resolvent_to_0}. As in the proof of Lemma \ref{lemma:continuity_resolvent_at_0}, define $T(\lambda) = - \lambda \log \left( \frac{\varepsilon}{4\vn{h}} \right)$. This leads to
	\begin{equation}\label{eqn:expression_Resolvent_to_0_upper_bound}
	R_n(\lambda)h_n(x_n) - h_n(x_n) \leq \frac{1}{2}\varepsilon + \int_0^{T(\lambda)} h(X_n(t)) - h(X_n(0)) \bQ_{n,\lambda} \circ \eta_n^{-1}(\dd X_n) \tau_{\lambda}(\dd t).
	\end{equation}
	Equation \eqref{eqn:control_on_entropies} yields that the restrictions of $\bQ_{n,\lambda} \circ \eta_n^{-1}$ to $\cF_T$ satisfy
	\begin{equation*}
	\sup_{\lambda } \sup_n \frac{1}{r_n} S_{T(\lambda)}(\bQ_{n,\lambda} \circ \eta_n^{-1} \, | \, \PR_{x_n}^n \circ \eta_n^{-1}) \leq 8 \vn{h}^2 \varepsilon^{-1}.
	\end{equation*}
	Denote as before $\widehat{\bQ}_{n,\lambda}$ the measure $\bQ_{n,\lambda}$ up to time $T(\lambda)$, concatenated with the Markovian kernel $\PR_y^n$. Thus,
	\begin{equation*}
	\sup_{\lambda } \sup_n \frac{1}{r_n} S_{1}(\widehat{\bQ}_{n,\lambda} \circ \eta_n^{-1} \, | \, \PR_{x_n}^n \circ \eta_n^{-1}) \leq 8 \vn{h}^2 \varepsilon^{-1}.
	\end{equation*}
	As the measures $\PR_{x_n}$ are exponentially tight, the measures $\widehat{\bQ}_{n,\lambda} \circ \eta_n^{-1}$ restricted to $\cF_1$ are tight due to Proposition \ref{proposition:exponential_tightness_and_entropy_control_implies_tightness}. Tightness implies we can control the modulus of continuity, which implies we can upper bound  \eqref{eqn:expression_Resolvent_to_0_upper_bound} uniformly in $n$ by $\varepsilon$ by choosing $\lambda$ small.
\end{proof}

\appendix

\section{Properties of relative entropy} \label{appendix:relative_entropy}

The following result by Donsker and Varadhan can be derived from Lemma's 4.5.8 and 6.2.13 of \cite{DZ98}.

\begin{lemma}[Donsker Varadhan] \label{lemma:DV_variational}
	Let $\cX$ be a Polish space. We have the following duality relations:
	\begin{align*}
	\log \ip{e^f}{\mu} & = \sup_{\bQ \in \cP(\cX)} \left\{\ip{f}{\nu} - S(\nu \, | \, \mu) \right\} & & \forall \, f \in C_b(\cX), \\
	S(\nu \, | \, \mu) & = \sup_{f \in C_b(\cX)} \left\{ \ip{f}{\nu} - \log \ip{e^f}{\mu} \right\} & & \forall \, \nu \in \cP(\cX).
	\end{align*}
\end{lemma}

By the second property of previous lemma, we immediately obtain lower semi-continuity of $S$.

\begin{lemma} \label{lemma:S_lsc}
	The map $(\nu,\mu) \rightarrow S(\nu \, | \, \mu)$ is lower semi-continuous.
\end{lemma}

%
%
%

We next give an extension of Theorem D.13 in \cite{DZ98}, given as Exercise 5.13 in \cite{RASe15}.

\begin{proposition} \label{proposition:relative_entropy_decomposition}
	Let $\cX$ be a Polish space. Suppose $\cF_1,\cF_2 \subseteq \cB(\cX)$ are such that $\cF_1 \otimes \cF_2= \cB(\cX)$ and $\cF_1 \cap \cF_2 = \{\emptyset,\cX\}$. Let $\mu,\nu$ be two Borel probability measures on $\cX$. Let $\mu_1,\nu_1$ denote the restrictions of $\mu$ and $\nu$ to $\cF_1$ and denote by $\mu(\cdot \, | \, x), \nu(\cdot \, | \, x)$ the regular conditional probabilities of $\mu$ and $\nu$ with respect to $\cF_1$.
	Then the map
	\begin{equation*}
	x \mapsto S(\nu(\cdot \, | \, x) \, | \, \mu(\cdot \, | \, x))
	\end{equation*}
	is measurable and we have
	\begin{equation*}
	S(\nu \, | \mu) = S(\nu_1 \, | \, \mu_1) + \int S(\nu(\cdot \, | \, x) \, | \, \mu(\cdot \, | \, x)) \, \nu_1(\dd x).
	\end{equation*}
	
\end{proposition}

The final result of this appendix is the equi-coercivity of relative entropy in the second component.

\begin{proposition} \label{proposition:equi_coercivity_relative_entropy}
	Let $\cX$ be a Polish space. Let $K \subseteq \cP(\cX)$ be a weakly compact set. For all $c \geq 0$ the set
	\begin{equation*}
	A := \bigcup_{\mu \in K} \left\{\nu \in \cP(\cX) \, \middle| \, S(\nu \, | \, \mu) \leq c \right\}
	\end{equation*}
	is compact for the weak topology on $\cP(\cX)$.
\end{proposition}

\begin{proof}
	First of all, note that by Lemma \ref{lemma:DV_variational} the set $A$ is weakly closed as $(\nu,\mu) \mapsto S(\mu \, | \, \nu)$ is lower semicontinuous. Thus, it suffices that $A$ is contained in a weakly compact set.
	
	\smallskip

	We start by proving that the map $p(f) := \sup_{\mu \in K} \log \ip{e^f}{\mu}$ is continuous for the strict topology. First note that the map $(f,\mu) \mapsto \log \ip{e^f}{\mu}$ is jointly continuous for the strict and weak topology. As a consequence $p$ is lower semi-continuous.
	
	Let $f_\alpha$ be a net that converges strictly to $f$. Let $\mu_\alpha$ be such that $p(f_\alpha) =  \log \ip{e^{f_\alpha}}{\mu_\alpha}$. By compactness there is a subnet $\mu_\beta$ that has a weak limit $\mu_0 \in K$. By joint continuity of $(f,\mu) \mapsto \log \ip{e^f}{\mu}$, we find
	\begin{equation*}
	\limsup_{\alpha \rightarrow \infty} p(f_\alpha) = \limsup_{\alpha \rightarrow \infty} \log \ip{e^{f_\alpha}}{\mu_\alpha} = \log \ip{e^f}{\mu_0} \leq p(f)
	\end{equation*}
	establishing that $p$ is upper semi-continuous and thus continuous.
	
	\smallskip
	
	Fix $c \geq 0$. Let $U := \left\{f \in C_b(\cX) \, \middle| \, p(f) \vee p(-f) \leq 1 \right\}$. As $p$ is strictly continuous, we find that $U$ is a strict neighbourhood of $0$.
	
	Let $\nu$ be such that there is $\mu_\nu$ with $S(\nu \, | \, \mu_\nu) \leq c$. For $f \in \cU$ we have by Young's inequality, or equivalently, Lemma \ref{lemma:DV_variational}, that
	\begin{equation*}
	|\ip{\nu}{f}| \leq \left(\log \ip{e^f}{\mu_\nu} \vee \log \ip{e^{-f}}{\mu_\nu} \right) + S(\nu \, | \, \mu_\nu) \leq 1 + c.
	\end{equation*}
	It follows that 
	\begin{equation*}
	A \subseteq \left\{\nu \, \middle| \, \forall \, f \in \cU :  |\ip{f}{\nu}| \leq 1+c\right\}
	\end{equation*}
	so that $A$ is contained in a weakly compact set by the Bourbaki-Alaoglu theorem, see e.g. \cite[Theorem 20.9.(4)]{Ko69}. 
	
	\smallskip
	
	Compactness of $A$ follows by the lower semi-continuity of $(\nu,\mu) \mapsto S(\nu \, | \, \mu)$, Lemma \ref{lemma:S_lsc}. 
\end{proof}

\section{Tightness for exponentially tilted measures}

The next technical result essentially states the following. Suppose that we have two sequences of measures $\PR_n$ and $\bQ_n$ such that
\begin{enumerate}
	\item the measures $\bP_n$ are exponentially tight at speed $r_n \rightarrow \infty$,
	\item there is some $M \geq 0$ such that $\sup_n \frac{1}{r_n} S(\bQ_n \, | \, \PR_n) \leq M$.
\end{enumerate}
Then the sequence $\bQ_n$ is tight.

\smallskip

The proposition below states this result, but allows for the context of a changing sequence of spaces and collections of measures instead of sequences.

\begin{proposition} \label{proposition:exponential_tightness_and_entropy_control_implies_tightness}
	Let $\cX_n$ be a collection of Polish spaces. For each $n$ let $A_n \subseteq \cP(\cX_n)$ and suppose that for each $a$ there is a collection of  compact sets $K_n  \subseteq \cX_n$ such that
	\begin{equation*}
	\limsup_n \frac{1}{r_n} \log \sup_{\mu \in A_n} \mu(K_n^c) \leq -a
	\end{equation*}
	for some sequence of constants $r_n \rightarrow \infty$. Suppose in addition that $B_n \subseteq \cP(\cX_n)$ are collections of measures such that there exists a constant $M \geq 0$ such that
	\begin{equation*}
	\sup_n \sup_{\nu \in B_n} \inf_{\mu \in A_n} \frac{1}{r_n} S(\nu \, | \, \mu) \leq M.
	\end{equation*}
	Then there is for each $\varepsilon > 0$ a constant $a = a(\varepsilon)$ such that
	\begin{equation*}
	\limsup_n \sup_{\nu \in B_n} \nu(K_n^c) \leq 2 \varepsilon.
	\end{equation*}
\end{proposition}

\begin{proof}
	Fix $\varepsilon > 0$, we will construct sets $K_n$ such that $\limsup_n \sup_{\nu \in B_n} \nu(K_n^c) \leq  2\varepsilon$ which suffices to establish the claim. Fix $a$ such that $a > M\varepsilon^{-1}$. By assumption there are sets $K_n  \subseteq \cX_n$ such that
	\begin{equation} \label{eqn:exp_tightness_proof_tightness}
	\limsup_n \sup_{\mu \in A_n} \frac{1}{r_n} \log \mu(K_{n}^c) \leq -a - 1.
	\end{equation}

	We introduce an auxiliary function in terms of which we can define relative entropy. Denote by $G : \bR^+ \rightarrow \bR^+$ the function 
	\begin{equation*}
	G(t) = \begin{cases}
	0 & \text{if } t = 0, \\
	t \log t & \text{if } t > 0.
	\end{cases}
	\end{equation*}
	Thus, we have by assumption that
	\begin{equation} \label{eqn:entropy_bound_in_tightness_proof}
	\sup_n \sup_{\nu \in B_n} \inf_{\mu \in A_n} \frac{1}{r_n} \int G\left( \frac{\dd \nu}{\dd \mu} \right) \dd \mu \leq M.
	\end{equation}

	Note that $G(t)t^{-1} =  \log(t)$ and that for $t \geq e^{r_nM \varepsilon^{-1}}$ we have that
	\begin{equation*}
	t \geq c_n := e^{r_n \frac{M}{\varepsilon}} \qquad \implies \qquad \frac{G(t)}{t} \geq \frac{G(c_n)}{c_n} = \log (c_n) = r_n \frac{M}{\varepsilon}.
	\end{equation*}
	Thus, for each $\nu \in B_n$ there is a $\mu \in A_n$ such that we obtain by \eqref{eqn:entropy_bound_in_tightness_proof} that
	\begin{equation} \label{eqn:excess_entropy_bound_in_tightness_proof}
	\int_{\left\{\frac{\dd \nu}{\dd \mu} \geq c_n\right\}} \frac{\dd \nu}{\dd \mu} \dd \nu  \leq \frac{\varepsilon}{r_nM} \int_{\left\{\frac{\dd \nu}{\dd \mu} \geq c_n\right\}} G\left(\frac{\dd \nu}{\dd \mu}\right) \dd \nu \leq \varepsilon.
	\end{equation}
	We use this inequality to bound $\nu_n[K_n^c]$. For sufficiently large $n$, we find by \eqref{eqn:exp_tightness_proof_tightness} and \eqref{eqn:excess_entropy_bound_in_tightness_proof} that
	\begin{align*}
	\nu(K_n^c) & = \int_{K_n^c \cap \left\{\frac{\dd \nu}{\dd \mu} \geq c_n\right\}} \frac{\dd \nu}{\dd \mu} \dd \mu +  \int_{K_n^c \cap \left\{\frac{\dd \nu}{\dd \mu} < c_n\right\}} \frac{\dd \nu}{\dd \mu} \dd \mu \\
	& \leq \varepsilon + \mu(K_n^c) c_n \\
	& \leq \varepsilon + e^{-r_na} c_n.
	\end{align*}
	As $a > M \varepsilon^{-1}$, the right-hand exponential term on the final line converges to $0$ independent from the choice of $\varepsilon$. This establishes the claim.
\end{proof}

\bibliographystyle{abbrv}
\bibliography{../KraaijBib}
\end{document}